\documentclass[11pt]{amsart}

\usepackage[a4paper,margin=1.08in]{geometry}
\usepackage{amsmath,amssymb,amsthm,mathtools,mathrsfs}
\usepackage{tikz-cd}
\usepackage{microtype}
\usepackage{enumitem}
\usepackage{hyperref}

\hypersetup{
  colorlinks=true,
  linkcolor=blue!40!black,
  citecolor=blue!40!black,
  urlcolor=blue!50!black
}

\allowdisplaybreaks
\setlist[enumerate]{leftmargin=2.2em,itemsep=0.15em,topsep=0.35em}

\newtheorem{thm}{Theorem}[section]
\newtheorem{prop}[thm]{Proposition}
\newtheorem{lem}[thm]{Lemma}

\theoremstyle{definition}
\newtheorem{defn}[thm]{Definition}
\newtheorem{rem}[thm]{Remark}

\numberwithin{equation}{section}

\newcommand{\bbC}{\mathbb{C}}
\newcommand{\bbR}{\mathbb{R}}
\newcommand{\cO}{\mathcal{O}}
\newcommand{\Nil}{\operatorname{Nil}}
\newcommand{\PI}{\operatorname{PI}}
\newcommand{\NI}{\operatorname{NI}}
\newcommand{\reg}{\mathrm{reg}}
\newcommand{\geom}{\mathrm{geom}}
\newcommand{\Ad}{\operatorname{Ad}}
\newcommand{\Tr}{\operatorname{Tr}}

\usepackage{graphicx} 

\title{Local Gan-Gross-Prasad conjecture for unitary groups: trace formula approach for epsilon dichotomy}

\begin{document}

\author{Cheng Chen}
\address{Department of Mathematics, University of Hong Kong, Hong Kong}
\email{chen5968@hku.hk}

\author{Xiaowei Li}
\address{School of Mathematical Sciences, Peking University, 5 Yi He Yuan Road,
Haidian District, Beijing, 100871, China.}
\email{lixiaowei26@stu.pku.edu.cn}

\author{Bushi Ren}
\address{School of Mathematical Sciences, Zhejiang University, Zijingang Campus, 866 Yuhangtang Road, Hangzhou, China
}
\email{3441239820@qq.com}

\maketitle

\begin{abstract}
    The purpose of this paper is to provide a proof independent of the theta-lifting for the epsilon dichotomy of the local Gan–Gross–Prasad conjecture of real unitary groups for tempered $L$-parameters. The approach is based on the local trace formula and endoscopic character identities following Waldspurger and Chen-Luo.%
\end{abstract}
\tableofcontents
\newpage
\section{Introduction}\label{sec:introduction}

  Let $E/F=\mathbb C/\mathbb R$ and let \(W\subset V\) be a relevant pair of Hermitian spaces in the sense of \cite{GGP12}, with
\[
  \dim_{\mathbb C}V=\dim_{\mathbb C}W+2t+1,
  \qquad t\in \mathbb{Z}_{\geq 0}.
\]
Associated with this pair is the Bessel subgroup
\[
  H=\mathrm U(W)\ltimes N\subset \mathrm U(V)\times \mathrm U(W),
\]together with its standard generic character $\nu$. For irreducible tempered representations \(\pi\) of
\(\mathrm U(V)\) and \(\sigma\) of \(\mathrm U(W)\), we write
\[
  m(\pi,\sigma)=m(\pi,\sigma)
  :=
  \dim \operatorname{Hom}_{H}
  \bigl(\pi\widehat\otimes\sigma,\nu\bigr).
\]

When $t=0$, the group $N$ is trivial, and the model becomes the branching problem from \(\mathrm U(V)\) to \(\mathrm U(W)\). If both unitary
groups are compact, the nonvanishing of this space is governed by the
classical interlacing rule for highest weights. For arbitrary real unitary
groups, one still has multiplicity at most one (\cite{SZ12}\cite{JSZ10}), while the local
Gan--Gross--Prasad conjecture predicts which member of a relevant Vogan
packet has nonzero multiplicity in terms of local root numbers.

The tempered conjecture for nonarchimedean unitary groups was proved by
Beuzart-Plessis using the local trace formula and endoscopy \cite{BP15}.
In the real case, his local trace formula gives a geometric expression for
the multiplicity and, as a consequence, the uniqueness of the distinguished
member inside a relevant tempered Vogan packet \cite{BP20}. The explicit
root-number description of that member was subsequently proved by Xue
\cite{Xue23}, by a different method based on theta correspondence.

The aim of the present paper is to recover the epsilon dichotomy directly
from the geometric multiplicity formula for the completeness of the
trace-formula approach to the tempered conjecture for real unitary groups. The organization of the
reduction is parallel to the endoscopic strategy used for real special
orthogonal groups in \cite{CL22}, but the geometry of Hermitian spaces and
real unitary endoscopy requires separate treatment.

We now state the main result. Fix the Whittaker normalization and the
additive character \(\psi_{\mathbb C}\) used in \cite{Ato20}, and write
\[
  \varepsilon(\rho,\psi_{\mathbb C})
  =
  \varepsilon\!\left(\frac12,\rho,\psi_{\mathbb C}\right).
\]
Let
\[
  \varphi=\bigoplus_{i=1}^{k}m_i\chi_i,
  \qquad
  \varphi'=\bigoplus_{j=1}^{\ell}n_j\mu_j
\]
be tempered parameters for \(\mathrm U(V)\) and \(\mathrm U(W)\),
respectively. After relabeling, suppose that
\(\chi_1,\ldots,\chi_r\), respectively
\(\mu_1,\ldots,\mu_s\), are precisely the distinct conjugate self-dual
characters having the sign required by the corresponding unitary group.
We write
\[
  A_{\varphi}
  =
  \bigoplus_{i=1}^{r}(\mathbb Z/2\mathbb Z)a_i,
  \qquad
  A_{\varphi'}
  =
  \bigoplus_{j=1}^{s}(\mathbb Z/2\mathbb Z)b_j.
\]

\begin{thm}[Tempered epsilon dichotomy]\label{thm:main}
Let \((V_{\alpha},W_{\alpha})\) range over the relevant pure inner forms
of \((V,W)\). Let
\[
  \pi=\pi(\varphi,\eta)\in\Pi_{\varphi}^{V_{\alpha}},
  \qquad
  \sigma=\pi(\varphi',\eta')\in\Pi_{\varphi'}^{W_{\alpha}}.
\]
Then
\[
  m_{V_{\alpha},W_{\alpha}}(\pi,\sigma)=1
\]
if and only if
\[
  \eta(a_i)
  =
  \varepsilon(\chi_i\otimes\varphi',\psi_{\mathbb C}),
  \qquad 1\leq i\leq r,
\]
and
\[
  \eta'(b_j)
  =
  \varepsilon(\varphi\otimes\mu_j,\psi_{\mathbb C}),
  \qquad 1\leq j\leq s.
\]
In particular, among all packet members on all relevant pure inner forms,
there is a unique distinguished pair.
\end{thm}

\subsubsection*{Outline of the proof}

The starting point is the geometric multiplicity formula of
Beuzart-Plessis \cite{BP20}. If \(\Theta\) is a tempered character, its
multiplicity is expressed as an integral over a space of semisimple
$H(\mathbb R)$-conjugacy classes, with the integrand determined by the regular
nilpotent germs of \(\Theta\). To make this formula compatible with
endoscopic transfer, we first describe the regular semisimple classes in
real unitary groups using parameter data $(\xi,c)$.
The stable class is determined by \(\xi\), whereas the sign vector \(c\)
records the ordinary conjugacy class within the stable class.

Over \(\mathbb C/\mathbb R\), the signature of the associated Hermitian
space is an explicit sum of the coordinates of \(c\). The Kottwitz sign is
governed, up to a factor depending only on the dimension, by their product.
This gives a concrete description of the union of regular semisimple
classes over pure inner forms with a prescribed Kottwitz sign. This
fixed-sign organization is the first specifically archimedean ingredient
in the argument.

We then combine this parametrization with Shelstad's endoscopic character
identities. For elements
\[
  s_V\in A_{\varphi},
  \qquad
  s_W\in A_{\varphi'},
\]
write the corresponding endoscopic decompositions as
\[
  \varphi
  =
  \mu_+\varphi_+\oplus\mu_-\varphi_-,
  \qquad
  \varphi'
  =
  \mu'_+\varphi'_+\oplus\mu'_-\varphi'_-.
\]
Let \(\Sigma_{V_{\pm}}\) and \(\Sigma_{W_{\pm}}\) denote the stable virtual
characters of the four endoscopic factors, and let \(S_{\lambda}\) denote
the stable geometric multiplicity with twisting character \(\lambda\).
The key geometric identity is the factorization
\[
\begin{aligned}
  \sum_{\alpha}
  m_{V_{\alpha},W_{\alpha}}
  \bigl(
     \Sigma^{s_V}_{V_{\alpha},\varphi},
     \Sigma^{s_W}_{W_{\alpha},\varphi'}
  \bigr)
  =
  \gamma^V_{\mu_+,\mu_-}\gamma^W_{\mu_+',\mu_-'}S_{\mu_+\mu'_-}
  \bigl(\Sigma_{V_+},\Sigma_{W_-}\bigr)
  \,
  S_{\mu_-\mu'_+}
  \bigl(\Sigma_{V_-},\Sigma_{W_+}\bigr),
\end{aligned}
\]
where the sum runs over the relevant pure inner forms.

This formula converts a weighted sum of the multiplicity function on $A_{\varphi}\times A_{\varphi'}$ into a product of stable terms attached to two crossed Gan--Gross--Prasad pairs. It is the mechanism that makes induction on the
rank possible. 

This approach is different from other cases in the following sense:
\begin{itemize}
    \item Over $\mathbb R$, the use of fixed Kottwitz signs is essential before summing
over the pure inner forms since a stable conjugacy class may contain ordinary conjugacy classes belonging to several different
signatures.
\item For unitary groups, one of the $S_{\mu_+\mu'_-}
  \bigl(\Sigma_{V_+},\Sigma_{W_-}\bigr)$ and $S_{\mu_-\mu'_+}
  \bigl(\Sigma_{V_-},\Sigma_{W_+}\bigr)$ corresponds to the multiplicity of the GGP model of Bessel type, while the other corresponds to the multiplicity the GGP model of Fourier-Jacobi type.
\end{itemize} 

The remainder of the proof consists of four reductions.

\begin{itemize}
    \item First, the compatibility of tempered parabolic induction with packets, multiplicities,
and epsilon factors removes all bad-parity constituents. We may therefore reduce to parameters of good
parity.

\item Second, suppose that a good-parity constituent occurs with a multiplicity of at least three. Two copies of that constituent can be removed by a parabolic reduction. 
identified with the original one. 
\item The third reduction is a pointwise endoscopic reduction to smaller GGP pairs of Bessel type. 
\item  The fourth reduction removes two distinct constituents from one endoscopic
block. 
\end{itemize}

Together with the character structure of the component groups, these
reductions imply that it is enough to consider
\[
  |A_{\varphi}|\leq 4,
  \qquad
  |A_{\varphi'}|\leq 4.
\]
Equivalently, each parameter has at most two distinct good-parity
constituents, and each constituent occurs with a multiplicity of at most two.
A final component-group argument reduces all remaining configurations to
one basic case.

In that basic case, one parameter is one-dimensional, and the other is the
sum of two distinct characters. The relevant packet consists of the
holomorphic and antiholomorphic discrete series of
\(\mathrm U(1,1)\), together with a finite-dimensional representation of
\(\mathrm U(0,2)\). Restricting these representations to
\(\mathrm U(1)\) gives three mutually exclusive ranges of weights. A
direct calculation of the corresponding epsilon factors selects exactly
the same representation in each range. This proves
Theorem~\ref{thm:main}.

\subsubsection*{Organization of the paper}

In Section~\ref{sec:geom-param}, we recall the geometric multiplicity
formula, parameterize regular semisimple classes in real unitary groups,
describe the signature and Kottwitz-sign constraints, and establish the
geometric endoscopic reduction formula.

In Section~\ref{sec:reduction-basic-cases}, we carry out the parabolic and endoscopic
reductions. We first reduce to good-parity parameters, then bound the
multiplicities of their irreducible constituents, and finally reduce the
sizes of the component groups to the basic configurations.

In Section~\ref{sec:basic-case}, we verify the remaining basic case by an
explicit branching calculation for \(\mathrm U(1)\) and compare the result
with the required local epsilon factors.

\subsubsection*{Acknowledgement} 

This project was completed during the undergraduate summer school of number theory and representation theory in August 2026, which is supported by the New Cornerstone Foundation at Peking University and the Shanghai Red Avenue Foundation. The first author is supported by the European Union’s Horizon 2020 research and innovation program under the Marie Skłodowska-Curie grant agreement No. 101034255.  The draft of this paper was produced from handwritten notes with GPT 5.6, and GPT Astra was used for proofreading. The proofs and the calculations in the article are not AI-assisted.   

\section{Notations and Preliminaries}
\label{sec:geom-param}

Throughout this section, let $(W,V)$ be a relevant pair of Hermitian spaces over
$\bbC/\bbR$.  For the relative GGP model, we write
\[
 G=U(V)\times U(W),\qquad H=U(W)\ltimes N.
\]
Thus $\Gamma(G,H)$ and the multiplicity $m(\ ,\ )$ refer to this relative pair.
When discussing conjugacy classes, packets, or endoscopic transfer, we write the two unitary factors $U(V)$ and $U(W)$ separately.  We recall the geometric
multiplicity formula and the regular semisimple parametrization used below; the parametrization is the real specialization of \cite[\S2]{BP15}.

\subsection{Relevant unitary GGP pairs}
\label{subsec:ggp-pairs}

We begin with the GGP datum, including the normalization of its Bessel
character. 

Put $n=\dim_{\bbC}W$.  Relevance means that there is an orthogonal
decomposition
\begin{equation}
  V=W\oplus^{\perp}Z,
  \qquad \dim_{\bbC}Z=2t+1,
  \label{eq:relevant-decomposition}
\end{equation}
in which $Z$ is quasi-split.  We choose a basis
\[
  z_{-t},\ldots,z_{-1},z_0,z_1,\ldots,z_t
\]
of $Z$ such that
\begin{equation}
  h_V(z_i,z_j)=(-1)^n\delta_{i,-j}.
  \label{eq:relevant-basis}
\end{equation}
Let $P\subset U(V)$ be the parabolic subgroup stabilizing the flag
\[
  \langle z_t\rangle\subset
  \langle z_t,z_{t-1}\rangle\subset\cdots\subset
  \langle z_t,\ldots,z_1\rangle,
\]
and let $N$ be its unipotent radical.  The diagonal copy of $U(W)$ in
$U(V)\times U(W)$ normalizes $N$.  We therefore set
\begin{equation}
  G=U(V)\times U(W),
  \qquad H=\Delta U(W)\ltimes N.
  \label{eq:ggp-groups}
\end{equation}
Here $(h,u)\in U(W)\ltimes N$ is embedded as $(hu,h)$ in $G$,
where $h$ acts trivially on $Z$.
For $u\in N(\bbR)$ define
\begin{equation}
  \nu(u)=
  \psi_{\bbR}\!\left(
    \operatorname{Tr}_{\bbC/\bbR}
    \sum_{i=0}^{t-1}h_V(z_{-i-1},uz_i)
  \right),
  \label{eq:bessel-character}
\end{equation}
and extend $\nu$ trivially across $\Delta U(W)$.  Changing the basis
subject to \eqref{eq:relevant-basis} gives an isomorphic GGP triple
$(G,H,\nu)$.  If $t=0$, then $N=1$ and this is the usual
branching model for $U(W)\subset U(V)$.

Let $\pi$ and $\sigma$ be irreducible tempered Casselman
–Wallach representations of $U(V)$ and $U(W)$, respectively, and put
$\Pi=\pi\widehat\otimes\sigma$.  The 
\emph{spectral GGP multiplicity} is
\begin{equation}
  m(\Pi)
  :=\dim\operatorname{Hom}_{H(\bbR)}
       (\Pi,\nu).
  \label{eq:spectral-multiplicity}
\end{equation}
In the notation of the Introduction,
\begin{equation}
  m(\pi,\sigma)
  :=m(\pi\widehat\otimes\sigma).
  \label{eq:mVW-convention}
\end{equation}
We extend $m$ bilinearly to finite virtual tempered representations.
The multiplicity-one theorems of \cite{JSZ10,SZ12} imply
$m(\pi,\sigma)\leq1$.

We shall sum over relevant pure inner forms.  If
$\alpha\in H^1(\bbR,U(W))$, let $W_\alpha$ be the corresponding Hermitian
space, set
\begin{equation}
  V_\alpha=W_\alpha\oplus^{\perp}Z,
  \label{eq:relevant-pure-inner-form}
\end{equation}
and define $G_\alpha$, $H_\alpha$, and $\nu_\alpha$ from
$(V_\alpha,W_\alpha)$ by the construction above.  Thus the index $\alpha$
always changes both spaces in a compatible way; it never denotes two
independent inner twists.

\subsection{The geometric multiplicity formula}
\label{subsec:geom-mult}

Let $H_{\mathrm{ss}}(\bbR)$ be the set of semisimple elements of $H(\bbR)$.  Every
$H(\bbR)$-conjugacy class in $H_{\mathrm{ss}}(\bbR)$ meets $U(W)$.  For a semisimple
$x\in U(W)$, put
\[
 W_x'=\ker(1-x\mid_ W),\qquad
 W_x''=\operatorname{Im}(1-x\mid_ W),\qquad
 V_x'=\ker(1-x\mid_ V).
\]
Then
\[
 W=W_x'\oplus^{\perp}W_x'',\qquad
 V=V_x'\oplus^{\perp}W_x''.
\]
Writing $U(W_x'')_x$ for the centralizer of $x$ in $U(W_x'')$, the connected
centralizers split accordingly as
\begin{align}
 G_x&=G_x'\times G_x'',
 &G_x'&=U(W_x')\times U(V_x'),
 &G_x''&=U(W_x'')_x\times U(W_x'')_x,\label{eq:centralizer-G}\\
 H_x&=H_x'\times H_x'',
 &H_x'&=U(W_x')\ltimes N_x,
 &H_x''&=U(W_x'')_x.\label{eq:centralizer-H}
\end{align}
The space $\Gamma(G,H)$ is the set of semisimple $H(\bbR)$-conjugacy classes
for which $H_x''$ is an anisotropic torus.  We use the topology and measure on
$\Gamma(G,H)$ defined in \cite[\S11.1]{BP20}.

For $x\in G_{\mathrm{ss}}(\bbR)$ set
\begin{equation}
 D^G(x)=\left|\det\!\left(1-\Ad(x)\mid_{
             \mathfrak g/\mathfrak g_x}\right)\right|.
 \label{eq:DG-def}
\end{equation}
If $x\in\Gamma(G,H)$, define
\begin{equation}
 \Delta(x)=D^G(x)D^H(x)^{-2}.
 \label{eq:Delta-def}
\end{equation}
Equivalently, in the present unitary setting,
\[
 \Delta(x)=\left|N_{\bbC/\bbR}
   \bigl(\det(1-x\mid W_x'')\bigr)\right|.
\]

Let $\Theta$ be a quasi-character on $G(\bbR)$.  For a semisimple $x$, its
Harish--Chandra germ expansion at $x$ has coefficients
$c_{\Theta,\cO}(x)$ indexed by nilpotent orbits
$\cO\subset\mathfrak g_x$.  We shall only need the regular nilpotent
coefficients.  Following \cite[\S11.2]{BP20}, put
\begin{equation}
 c_\Theta(x)=
 \begin{cases}
 \displaystyle
 \frac{1}{|\Nil_{\reg}(\mathfrak g_x)|}
 \sum_{\cO\in\Nil_{\reg}(\mathfrak g_x)}c_{\Theta,\cO}(x),
 & \Nil_{\reg}(\mathfrak g_x)\neq\varnothing,\\[4mm]
 0,&\Nil_{\reg}(\mathfrak g_x)=\varnothing.
 \end{cases}
 \label{eq:cTheta-def}
\end{equation}
Thus $c_\Theta(x)=0$ when $G_x$ is not quasi-split. Denote $\Gamma^{\text{qs}}(G,H)\subset \Gamma(G,H)$ the subspace of elements which centralizer in G is quasi-split.

\begin{defn}\label{def:mgeom}
The geometric multiplicity of a quasi-character $\Theta$ is
\begin{equation}
 m_{\geom}(\Theta)
 =\lim_{s\to0^+}\int_{\Gamma(G,H)}
 D^G(x)^{1/2}c_\Theta(x)\Delta(x)^{s-1/2}\,dx.
 \label{eq:mgeom-def}
\end{equation}
The integral is absolutely convergent for $\Re(s)>0$, and the limit in
\eqref{eq:mgeom-def} exists; see \cite[Prop.~11.2.1]{BP20}.
\end{defn}

If $\pi$ is a tempered representation of $G(\bbR)$ with Harish--Chandra
character $\Theta_\pi$, set
\[
 m_{\geom}(\pi)=m_{\geom}(\Theta_\pi).
\]
The local trace formula of Beuzart-Plessis gives the following identity.

\begin{thm}[Beuzart-Plessis]\label{thm:BP-geometric-multiplicity}
For every tempered representation $\pi$ of $G(\bbR)$,
\begin{equation}
 m_{\geom}(\pi)=m(\pi).
 \label{eq:BP-mgeom=m}
\end{equation}
\end{thm}

\section{Parametrization of the space $\Gamma(G,H)$}

\subsection{Regular semisimple parameters for unitary groups}
\label{subsec:unitary-regular-parameters}

We now recall only the part of the regular semisimple parametrization that will
be used below.  In particular, we do not attempt to rewrite the full formalism
of \cite[\S2]{BP15}.

Let $E=\bbC$ and $F=\bbR$.  A parameter datum is a tuple
\begin{equation}
 \xi=\bigl(I,(F_{\pm i})_{i\in I},(F_i)_{i\in I},(y_i)_{i\in I}\bigr)
 \label{eq:xi-datum}
\end{equation}
with the following properties:
\begin{enumerate}
 \item $I$ is a finite set;
 \item each $F_{\pm i}$ is a finite field extension of $F$, and
       \[
       F_i=F_{\pm i}\otimes_F E
       \]
       is viewed as a quadratic \`etale $F_{\pm i}$-algebra; we denote by
       $\tau_i$ the involution induced by complex conjugation on $E$;
 \item $y_i\in F_i^\times$ satisfies
       \begin{equation}
       y_i\tau_i(y_i)=1.
       \label{eq:yi-unitary}
       \end{equation}
\end{enumerate}
Two such data are identified if a bijection of their index sets can be equipped
with compatible $F$-algebra isomorphisms carrying every $y_i$ to the
corresponding $y_i'$.  Let $\Xi$ denote the set of isomorphism classes of these
data.  We call $\xi$ \emph{regular} if it has no nontrivial automorphism, and
write $\Xi_{\reg}$ for the regular locus.  For $d\geq0$ put
\begin{equation}
 \Xi_{\reg,d}
 =\left\{\xi\in\Xi_{\reg}:\sum_{i\in I}[F_i:E]=d\right\}.
 \label{eq:Xi-reg-d}
\end{equation}

Associated with $\xi$ is the finite abelian group
\begin{equation}
 C(\xi)=\prod_{i\in I}
 F_{\pm i}^{\times}/
 {N_{F_i/F_{\pm i}}(F_i^\times)}.
 \label{eq:Cxi-def}
\end{equation}
Let
\[
 I_\xi^*=\{i\in I:F_i\text{ is a field}\}.
\]
Over $\bbR$ the quotient in \eqref{eq:Cxi-def} is nontrivial exactly for
$i\in I_\xi^*$, and the sign map identifies
\begin{equation}
 C(\xi)\simeq\prod_{i\in I_\xi^*}\{\pm1\}.
 \label{eq:Cxi-signs}
\end{equation}

Choose $c=(c_i)_{i\in I}$ with $c_i\in F_{\pm i}^{\times}$ representing an
element of $C(\xi)$.  Set
\begin{equation}
 V_{\xi,c}=\bigoplus_{i\in I}F_i
 \label{eq:Vxic-space}
\end{equation}
and equip it with the $E/F$-Hermitian form
\begin{equation}
 h_{\xi,c}\bigl((w_i)_i,(w_i')_i\bigr)
 =\sum_{i\in I}\Tr_{F_i/E}
       \bigl(c_i\,\tau_i(w_i)w_i'\bigr).
 \label{eq:Vxic-form}
\end{equation}
Its isomorphism class depends only on $\xi$ and the class of $c$ in $C(\xi)$.
Moreover, multiplication by $y_i$ on the $i$-th summand defines an isometry
\begin{equation}
 x_{\xi,c}(w_i)_{i\in I}=(y_iw_i)_{i\in I}
 \in U(V_{\xi,c})(\bbR),
 \label{eq:x-xi-c}
\end{equation}
where \eqref{eq:yi-unitary} is exactly the condition that ensures
$x_{\xi,c}$ preserves $h_{\xi,c}$.

\subsection{The real signature condition}
\label{subsec:real-signature-condition}

The preceding construction becomes particularly concrete over
$\bbC/\bbR$.  There are only two possibilities for $F_{\pm i}$.

If $F_{\pm i}=\bbR$, then $F_i=\bbC$ and the corresponding summand in
\eqref{eq:Vxic-space} has complex dimension one.  Under
\eqref{eq:Cxi-signs}, the sign $c_i=+1$ gives a positive line and $c_i=-1$
gives a negative line.  If $F_{\pm i}=\bbC$, then
\[
 F_i\simeq\bbC\times\bbC,
\]
and this summand has Hermitian signature $(1,1)$; moreover, its contribution to
$C(\xi)$ is trivial.

Write
\begin{align*}
 I_{\bbR}^+&=\{i:F_{\pm i}=\bbR,\ c_i=+1\},\\
 I_{\bbR}^-&=\{i:F_{\pm i}=\bbR,\ c_i=-1\},\\
 I_{\bbC}&=\{i:F_{\pm i}=\bbC\}.
\end{align*}
The signature of $V_{\xi,c}$ is therefore
\begin{equation}
 \PI(V_{\xi,c})=|I_{\bbR}^+|+|I_{\bbC}|,
 \qquad
 \NI(V_{\xi,c})=|I_{\bbR}^-|+|I_{\bbC}|.
 \label{eq:signature-xi-c}
\end{equation}
In particular,
\begin{equation}
 \PI(V_{\xi,c})-\NI(V_{\xi,c})
   =\sum_{i\in I_\xi^*}c_i.
 \label{eq:signature-difference}
\end{equation}
This is the unitary analog of the signature calculation in
\cite[\S3.1]{CL22}; the quadratic-space dimensions occurring there are replaced
here by Hermitian dimensions.

Let $V$ be a nondegenerate Hermitian space of complex dimension $d$, and set
\begin{equation}
 \Delta_V=\PI(V)-\NI(V).
 \label{eq:Delta-V}
\end{equation}
For $\xi\in\Xi_{\reg,d}$, define
\begin{equation}
 C(\xi)_{\Delta_V}
 =\left\{c\in C(\xi):
       \sum_{i\in I_\xi^*}c_i=\Delta_V\right\}
 \label{eq:Cxi-Delta}
\end{equation}
and
\begin{equation}
 \Xi_{\reg,V}
 =\left\{(\xi,c):
       \xi\in\Xi_{\reg,d},\ c\in C(\xi)_{\Delta_V}\right\}.
 \label{eq:Xi-reg-V}
\end{equation}
Equation \eqref{eq:signature-difference} shows that
\[
 (V_{\xi,c},h_{\xi,c})\simeq V
 \quad\Longleftrightarrow\quad
 c\in C(\xi)_{\Delta_V}.
\]
Consequently, the regular semisimple conjugacy classes in $U(V)$ admit the
following description.

\begin{prop}[Regular semisimple parametrization]\label{prop:unitary-rss-param}
The assignment $(\xi,c)\mapsto[x_{\xi,c}]$ induces a bijection
\begin{equation}
 \Xi_{\reg,V}
 \xrightarrow{\ \sim\ }
 U(V)_{\reg}(\bbR)/\mathrm{conj}.
 \label{eq:unitary-rss-bijection}
\end{equation}
Under this bijection, the stable conjugacy class is determined by $\xi$, while
$c$ records the ordinary conjugacy class within that stable class.
\end{prop}

\begin{proof}
This is the specialization to $F=\bbR$, $E=\bbC$ of the regular semisimple
parametrization in \cite[\S2]{BP15}.  The only additional point is the explicit
signature constraint, which follows immediately from
\eqref{eq:signature-xi-c}.
\end{proof}

\subsection{Kottwitz signs and fixed-sign parameter spaces}
\label{subsec:kottwitz-param}

We next refine the parametrization above by keeping track of the Kottwitz sign,
following the corresponding organization in \cite[\S\S3.2,4.3]{CL22}.

\begin{defn}\label{def:kottwitz-sign}
Let $G$ be a connected reductive group over $\bbR$, let $G_{\mathrm{qs}}$ be its
quasi-split inner form, and choose maximal compact subgroups
$K\subset G(\bbR)$ and $K_{\mathrm{qs}}\subset G_{\mathrm{qs}}(\bbR)$.  The Kottwitz sign is
\begin{equation}\label{eq:kottwitz-sign}
 e(G)=(-1)^{\frac{\dim K_{\mathrm{qs}}-\dim K}{2}}.
\end{equation}
\end{defn}

For a Hermitian space of signature $(p,q)$, this gives the following explicit
formula.

\begin{lem}\label{lem:kottwitz-unitary-sign}
For $U(p,q)$ one has
\begin{equation}\label{eq:kottwitz-unitary}
 e\bigl(U(p,q)\bigr)=
 \begin{cases}
  1, & p+q\ \text{odd},\\[2mm]
  (-1)^{\frac{(p-q)^2}{4}}, & p+q\ \text{even}.
 \end{cases}
\end{equation}
\end{lem}

For the rest of this subsection, we keep only the notation needed below.  For
$\xi\in\Xi_{\reg}$ and $\varepsilon\in\{\pm1\}$, write
\begin{equation}\label{eq:Cxi-pm}
 C(\xi)^\varepsilon
 :=\left\{c=(c_i)\in C(\xi):
       \prod_{i\in I_\xi^*}c_i=\varepsilon\right\}.
\end{equation}

We record the following three parametrization statements. Their proofs are just simple computations, we omit them in this paper.

\begin{prop}\label{prop:fixed-sign-one-unitary}
Let $V$ be a nondegenerate Hermitian space of even dimension $d_V$, and fix
$e_0\in\{\pm1\}$.  Then the parametrization above induces a bijection
\begin{equation}\label{eq:fixed-sign-one-unitary}
 \bigsqcup_{\substack{\alpha\in H^1(\bbR,U(V))\\
              e(U(V_\alpha))=e_0}}
 \Xi_{\reg,V_\alpha}
 \xrightarrow{\ \sim\ }
 \left\{(\xi,c):
   \xi\in\Xi_{\reg,d_V},\ 
   c\in C(\xi)^{(-1)^{\frac{I_\xi^*}{2}}e_0}\right\}.
\end{equation}
\end{prop}

We now return to the relevant pair $(V,W)$.  Put
$d_V=\dim V$, $d_W=\dim W$, and
$\Delta_W=\PI(W)-\NI(W)$.  Let $\mathcal C(V,W)$ be the sets of $(\kappa'',c'')\in \Xi_{\reg}$ such that 
$$
(W_{\kappa'',c''},q_{\kappa'',c''})\hookrightarrow (W,q_W)\qquad
(W_{\kappa'',c''},q_{\kappa'',c''})\hookrightarrow (V,q_V)
$$
are the orthogonal subspaces, the orthogonal complement are quasi-split respectively, and $I_{\kappa''}=I_{\kappa''}^*$.

Thus, the parametrization above gives a bijection
\begin{equation}\label{eq:Gamma-C-bijection}
 \Gamma(G,H)^{\text{qs}}\xrightarrow{\ \sim\ }\mathcal C(V,W).
\end{equation}
There is a natural projection
\begin{equation}\label{eq:C-to-Xi-projection}
 \mathcal C(V,W)\longrightarrow\Xi(d_V,d_W),
 \qquad (\xi,c)\longmapsto\xi,
\end{equation}
where 

\begin{align}\label{1}
    \Xi(d_V,d_W)=\{\xi''\in \Xi_{\mathrm{reg}}~|~I_{\xi''}=I^*_{\xi''},|I^*_{\xi''}|\le min\{d_V,d_W\},\}
\end{align}

Then, we have the following two propositions to describe the fibers of the \eqref{eq:C-to-Xi-projection} and the fibers of the pure inner form version of \eqref{eq:C-to-Xi-projection}.

\begin{prop}\label{prop:fiber-CVW}
For $\xi''\in\Xi(d_V,d_W)$, the fiber of
$\mathcal C(V,W)\to\Xi(d_V,d_W)$ over $\xi''$ is
\begin{equation}\label{eq:fiber-CVW}
 \begin{cases}
 C(\xi'')_{
   \Delta_W-\frac{1+(-1)^{|I_{\xi''}^*|}}{2}},
      & d_W\ \text{odd},\\[3mm]
 C(\xi'')_{
   \Delta_W+\frac{1-(-1)^{|I_{\xi''}^*|}}{2}},
      & d_W\ \text{even}.
 \end{cases}
\end{equation}
\end{prop}

For a pure inner form indexed by
$\alpha\in H^1(\bbR,U(W))$, we use the same notation
$\mathcal C(V_\alpha,W_\alpha)$.  The preceding two propositions have the
following fixed-sign counterpart.

\begin{prop}\label{prop:fiber-fixed-kottwitz}
Fix $e_0\in\{\pm1\}$.  For $\xi''\in\Xi(d_V,d_W)$, the fiber over $\xi''$ of
\begin{equation}\label{eq:fixed-family-map}
 \bigsqcup_{\substack{\alpha\in H^1(\bbR,U(W))\\
 e(U(W_\alpha)\times U(V_\alpha))=e_0}}
 \mathcal C(V_\alpha,W_\alpha)
 \longrightarrow\Xi(d_V,d_W)
\end{equation}
is
\begin{equation}\label{eq:fixed-family-fiber}
 C(\xi'')^{\,(-1)^{\lfloor d_{\xi''}/2\rfloor}e_0}.
\end{equation}
\end{prop}

\section{Geometric endoscopy reduction formula}

\subsection{Endoscopic character transfer}
\label{subsec:endoscopic-character-transfer}

Let $\phi$ and $\phi'$ be tempered parameters for $U(V)$ and $U(W)$.
Denote by $z_\phi\in A_\phi$ and $z_{\phi'}\in A_{\phi'}$ the images of
the scalar $-1$ in the corresponding component groups.  For
$s_V\in A_\phi$ and a pure inner form $U(V_\alpha)$, put
\begin{equation}\label{eq:virtual-packet-character-total}
 {\Sigma}^{s_V}_{V_\alpha,\phi}
 :=\sum_{\pi\in\Pi^{V_\alpha}_\phi}\eta_\pi(s_V)\,\pi,
 \qquad
 \Theta_{{\Sigma}^{s_V}_{V_\alpha,\phi}}
 :=\sum_{\pi\in\Pi^{V_\alpha}_\phi}\eta_\pi(s_V)\Theta_\pi,
\end{equation}
and define ${\Sigma}^{s_W}_{W_\beta,\phi'}$ and its distribution
character in the same way.

We use the following consequences of Shelstad's endoscopic character theory.
\begin{enumerate}
 \item The distribution characters
 $\Theta_{{\Sigma}_{V_+}}$,
 $\Theta_{{\Sigma}_{V_-}}$,
 $\Theta_{{\Sigma}_{W_+}}$, and
 $\Theta_{{\Sigma}_{W_-}}$
 are stable; see \cite[Lemma~5.2]{She79}.
 \item With the Whittaker-normalized transfer factors,
 
 \[ e\!\left(U(V_\alpha)\right)
   \Theta_{{\Sigma}^{s_V}_{V_\alpha,\phi}},\qquad
   e\!\left(U(W_\beta)\right)
   \Theta_{{\Sigma}^{s_W}_{W_\beta,\phi'}}
 \]
 are the endoscopic transfer of $\Theta_{{\Sigma}_{V_+}}\otimes
  \Theta_{{\Sigma}_{V_-}}$ and 
 $\Theta_{{\Sigma}_{W_+}}\times
  \Theta_{{\Sigma}_{W_-}}$ respectively.
 We use these identities in the form provided by Shelstad's real endoscopic
 transfer theory; see \cite{She81,She82,She08a,She08c,She12}.
\end{enumerate}

\subsection{Stable geometric multiplicities and endoscopic reduction}
\label{subsec:stable-geometric-endoscopy}

We finally introduce the stable geometric term that occurs in the endoscopic
reduction.  We keep the notation of \cite[\S6]{BP15}.  Let $E/F=\bbC/\bbR$ 
and $d,d'\geq0$ be the dimensions of the Hermitian spaces, respectively, and let $\Xi(d,d')$ denote the elliptic parameter space as \eqref{1}, which adds the condition that $|I^*_{\xi''}|$ has the same parity as $d$ or $d'$.  For
$\xi\in\Xi(d,d')$, write $d_\xi$ for its degree and $P_\xi$ for the
associated characteristic polynomial.  Fix $\delta\in E^\times$ with
$\operatorname{Tr}_{E/F}(\delta)=0$. In this paper, we set $\delta=i$.

\begin{defn}\label{def:Smu}
Let $\mu:E^\times\to\bbC^\times$ be a continuous character satisfying
\begin{equation}\label{eq:mu-restriction-Smu}
 \mu\big|_{F^\times}=\operatorname{sgn}_{E/F}^{\,d+d'+1}.
\end{equation}
For stable quasi-characters ${\Sigma}$ and ${\Sigma}'$ of the
corresponding unitary groups, define
\begin{align}
 S_\mu({\Sigma},{\Sigma}')
 :=\lim_{s\to0^+}\int_{\Xi(d,d')}
 &\mu\!\left(\delta^{-d_\xi}P_\xi(1)\right)
 |C(\xi)|\,
 c_{{\Sigma}}(\xi)c_{{\Sigma}'}(\xi)\notag\\
 &\times D^d(\xi)^{1/2}D^{d'}(\xi)^{1/2}
 \Delta(\xi)^{-1/2+s}\,d\xi.
 \label{eq:Smu-definition}
\end{align}
\end{defn}

We now state the geometric endoscopic reduction formula in the form needed
later. This is an archimedean analogy of \cite[Proposition~6.2.1]{BP15}.

\begin{thm}[Geometric endoscopic reduction formula]
\label{thm:geometric-endoscopy-reduction}
Let $(V,W)$ be a relevant pair of nondegenerate Hermitian spaces over
$E/F=\bbC/\bbR$.  Suppose that the endoscopic data give decompositions
\[
 V=V_+\oplus^\perp V_-,
 \qquad
 W=W_+\oplus^\perp W_-.
\]
Put
\[
 d_\pm=\dim V_\pm,
 \qquad
 d'_\pm=\dim W_\pm.
\]
Choose the endoscopic characters $\mu_+,\mu_-,\mu'_+,\mu'_-$ so that
\begin{equation}\label{eq:endoscopic-mu-restrictions-geom}
 \mu_+|_{F^\times}=\operatorname{sgn}_{E/F}^{d_-},\qquad
 \mu_-|_{F^\times}=\operatorname{sgn}_{E/F}^{d_+},\qquad
 \mu'_+|_{F^\times}=\operatorname{sgn}_{E/F}^{d'_-},\qquad
 \mu'_-|_{F^\times}=\operatorname{sgn}_{E/F}^{d'_+}.
\end{equation}
Let
${\Sigma}_{V_\pm}$ and ${\Sigma}_{W_\pm}$ be the stable
quasi-characters on the four endoscopic factors, and let
${\Sigma}^{s_V}_{V_\alpha}$ and
${\Sigma}^{s_W}_{W_\alpha}$ denote their endoscopic transfers to the
pure inner form indexed by $\alpha$.  Then, for every $e_0\in\{\pm1\}$,
\begin{align}
 &\sum_{\substack{\alpha\in H^1(\bbR,U(W))\\ e(G_\alpha)=e_0}}
 m\!\left({\Sigma}^{s_W}_{W_\alpha},
          {\Sigma}^{s_V}_{V_\alpha}\right)
 =\frac12 \gamma^V_{\mu_+,\mu_-}\gamma^W_{\mu_+',\mu_-'}\Bigl(
 e_0\,
 S_{\mu_+\mu'_+}({\Sigma}_{V_+},{\Sigma}_{W_+})
 S_{\mu_-\mu'_-}({\Sigma}_{V_-},{\Sigma}_{W_-})
 \notag\\
 &\hspace{31mm}+
 S_{\mu_+\mu'_-}({\Sigma}_{V_+},{\Sigma}_{W_-})
 S_{\mu_-\mu'_+}({\Sigma}_{V_-},{\Sigma}_{W_+})
 \Bigr).
 \label{eq:geometric-endoscopy-reduction}
\end{align}
where
$$
\gamma^V_{\mu_+,\mu_-}=-i(\text{$d_+$ and $d_-$ are all odd}),1(\text{otherwise})\qquad \gamma^W_{\mu_+',\mu_-'}=-i(\text{$d_+'$ and $d_-'$ are all odd}),1(\text{otherwise})
$$
\end{thm}

\begin{proof}
    I will give a sketch proof there, more details could refer to \cite{BP15}.

First, the left hand of the equation could be written as
$$
\lim_{s\to 0}\int_{\Xi(d_V,d_w)}f_{1,e_0}(\xi)\Delta(\xi)^{s-\frac{1}{2}}d\xi
$$
where
$$
f_{1,e_0}(\xi)=D^{d_V}(\xi)^{\frac{1}{2}} D^{d_W}(\xi)^{\frac{1}{2}}\sum_{c\in C(\xi)^{(-1)^{[\frac{d_\xi}{2}]}e_0}}c_{\Sigma^{s_V}_{V_\alpha(\xi,c)}}(x(\xi,c))c_{\Sigma^{s_W}_{W_\alpha(\xi,c)}}(x(\xi,c))
$$

And the right hand of the equation is just
$$
\lim_{s\to 0}\int\limits_{\Xi(d_V,d_W)}f_{2,e_0}(\xi)\Delta(\xi)^{s-\frac{1}{2}}d\xi
$$
where 

\begin{align}
\begin{aligned}
f_{2,e_0}(\xi)
={}&\frac{\gamma^V_{\mu_+,\mu_-}
          \gamma^W_{\mu'_+,\mu'_-}}{2}
|C(\xi)|\\
&\times\Biggl(e_0
\sum_{(\xi_1,\xi_2)=\xi}
(\mu_+\mu'_+)\bigl(\delta^{-d_{\xi_1}}P_{\xi_1}(1)\bigr)
c_{\Sigma_{V_+}}(\xi_1)c_{\Sigma_{W_+}}(\xi_1)
D^{d_+}(\xi_1)^{\frac{1}{2}}D^{d'_+}(\xi_1)^{\frac{1}{2}}\\
&\qquad\qquad\times
(\mu_-\mu'_-)\bigl(\delta^{-d_{\xi_2}}P_{\xi_2}(1)\bigr)
c_{\Sigma_{V_-}}(\xi_2)c_{\Sigma_{W_-}}(\xi_2)
D^{d_-}(\xi_2)^{\frac{1}{2}}D^{d'_-}(\xi_2)^{\frac{1}{2}}\\
&\quad+
\sum_{(\xi_1,\xi_2)=\xi}
(\mu_+\mu'_-)\bigl(\delta^{-d_{\xi_1}}P_{\xi_1}(1)\bigr)
c_{\Sigma_{V_+}}(\xi_1)c_{\Sigma_{W_-}}(\xi_1)
D^{d_+}(\xi_1)^{\frac{1}{2}}D^{d'_-}(\xi_1)^{\frac{1}{2}}\\
&\qquad\qquad\times
(\mu_-\mu'_+)\bigl(\delta^{-d_{\xi_2}}P_{\xi_2}(1)\bigr)
c_{\Sigma_{V_-}}(\xi_2)c_{\Sigma_{W_+}}(\xi_2)
D^{d_-}(\xi_2)^{\frac{1}{2}}D^{d'_+}(\xi_2)^{\frac{1}{2}}
\Biggr).
\end{aligned}
\end{align}

Now, we aim to prove that $f_{1,e_0}(\xi)=f_{2,e_0}(\xi)$, we utilize the perturbation formula to convert function $c_{\theta}$ to the ordinary character \cite[Proposition 4.5.1]{BP20}. And The endoscopic character relation formula in \eqref{subsec:endoscopic-character-transfer} is the basic endoscopic input.
$$
D^G(x)^{\frac{1}{2}}c_{\theta}(x)=|W(G_x,T_{qd,x})|^{-1}\lim_{x'\in T_{qd,x}(F)\to x}D^G(x')^{\frac{1}{2}}\theta(x')
$$
$$
e(G_\alpha)\,D^{G_\alpha}(x)^{\frac{1}{2}}
\Theta_{\Sigma_\alpha^{s}}(x)
=
\sum_{\substack{
y\in H(\mathbb R)_{\mathrm{reg}}/\mathrm{stconj}\\
y\leftrightarrow x}}
\Delta_{\mathrm{Wh}}(y,x)\,
D^H(y)^{\frac{1}{2}}\Theta_{\Sigma_H}(y)
$$

At last, the computation about the multiplication of the transfer factor and the linear combination about the fibers could produce the last theorem \cite{BP15}, in particular the normalization factor in \cite{BP15} is just $\gamma^V_{\mu_+,\mu_-}$ there \cite{KS12}.
\end{proof}

Adding the identities of Theorem~\ref{thm:geometric-endoscopy-reduction} for
$e_0=1$ and $e_0=-1$ gives
\begin{equation}\label{eq:endoscopic-reduction-summed}
 \sum_{\alpha\in H^1(\bbR,U(W))}
 m\!\left({\Sigma}^{s_W}_{W_\alpha,\phi'},
          {\Sigma}^{s_V}_{V_\alpha,\phi}\right)
 =\gamma^V_{\mu_+,\mu_-}\gamma^W_{\mu_+',\mu_-'}
 S_{\mu_+\mu'_-}({\Sigma}_{V_+},{\Sigma}_{W_-})
 S_{\mu_-\mu'_+}({\Sigma}_{V_-},{\Sigma}_{W_+}).
\end{equation}
If one of the two pairs occurring in this product is $(0,0)$, we use
\begin{equation}\label{eq:S00}
 S_\nu(0,0)=1.
\end{equation}

We now recover the stable geometric term by the special component-group
elements used above.  Let $\mu$ satisfy
\eqref{eq:mu-restriction-Smu}.  Take
\[
  s_V=z_\phi,\qquad s_W=1,
\]
and choose the endoscopic characters so that $\mu_-\mu'_+=\mu$.  For these
choices,
\[
  V_+=0,\qquad V_-=V,\qquad W_+=W,\qquad W_-=0.
\]
Therefore \eqref{eq:endoscopic-reduction-summed} and
\eqref{eq:S00} give
\begin{equation}\label{eq:Smu-special-component-elements}
 S_\mu({\Sigma}_V,{\Sigma}_W)
 =\sum_{\alpha\in H^1(\bbR,U(W))}
 m\!\left(
   {\Sigma}^{\,1}_{W_\alpha,\phi'},
   {\Sigma}^{\,z_\phi}_{V_\alpha,\phi}
 \right).
\end{equation}
Equivalently, one may take $s_V=1$ and $s_W=z_{\phi'}$ and choose
$\mu_+\mu'_- = \mu$.

\section{Basic reduction method and the reduction to the basic cases}\label{sec:reduction-basic-cases}

In this section, we reduce Conjecture~2 to the basic cases.  The argument has
four steps: parabolic reduction to good parity, reduction of the multiplicities
of good parity constituents, reduction by removing two distinct constituents from different endoscopic
blocks, and a reduction by removing two distinct constituents from one endoscopic
block.  The organization follows the reduction strategy of \cite[\S5]{CL22}.
For unitary endoscopy, we use the conventions of \cite{Mok15,BP15}.

\medskip

\subsection{The first reduction}

We first consider $(V,\phi)$.  Let
\[
  M_V:=\operatorname{std}\circ\phi\big|_{W_{\mathbb C}}
\]
be the representation of $W_{\mathbb C}=\mathbb C^\times$ associated with
$\phi$.  We decompose it as
\begin{equation}\label{eq:good-bad-parity-decomp}
  M_V=M_{V,\mathrm{gp}}
      \oplus\bigl(M_{V,\mathrm{bp}}
      \oplus M_{V,\mathrm{bp}}^{c,\vee}\bigr),
\end{equation}
where the first term is the good-parity part and the second term is the sum of
all paired bad-parity constituents.  Choose a corresponding orthogonal
decomposition
\begin{equation}\label{eq:hermitian-parabolic-decomp}
  V=V_{\mathrm{gp}}
    \oplus^{\perp}\bigl(X_{\mathrm{bp}}\oplus X_{\mathrm{bp}}^{\vee}\bigr),
\end{equation}
where $X_{\mathrm{bp}}$ is totally isotropic.  If $P_X$ is the parabolic
subgroup stabilizing $X_{\mathrm{bp}}$, then its Levi factor is
\[
  M_X\simeq
  \operatorname{Res}_{\mathbb C/\mathbb R}\mathrm{GL}(X_{\mathrm{bp}})
  \times U(V_{\mathrm{gp}}).
\]
Accordingly, $\phi$ factors through ${}^LM_X$, and we denote the two factors by
$\phi_{\mathrm{bp}}^{\mathrm{GL}}$ and $\phi_{\mathrm{gp}}$.

The paired summand in \eqref{eq:good-bad-parity-decomp} contributes no new
generator to the component group. Hence, there is a natural identification
\begin{equation}\label{eq:component-group-parabolic}
  A_{\phi_{\mathrm{gp}}}\xrightarrow{\sim}A_\phi.
\end{equation}
Let $\tau_{\mathrm{bp}}$ be the tempered representation of
$\mathrm{GL}(X_{\mathrm{bp}})(\mathbb C)$ attached to
$\phi_{\mathrm{bp}}^{\mathrm{GL}}$.  Under the local Langlands correspondence,
normalized parabolic induction gives
\begin{equation}\label{eq:packet-parabolic-bijection}
  \pi_{\mathrm{gp}}\longmapsto
  \operatorname{Ind}_{P_X}^{U(V)}
  \bigl(\tau_{\mathrm{bp}}\boxtimes\pi_{\mathrm{gp}}\bigr).
\end{equation}
The induced representation is irreducible and tempered, and the character of
the component group is unchanged under
\eqref{eq:component-group-parabolic}; for the parabolic-induction description
of real unitary packets, see \cite[Theorem~5.3]{MR19}.
The compatibility of the Gross--Prasad multiplicity with the corresponding
parabolic induction is the tempered intertwining result
\cite[Proposition~7.4.1]{BP20}.

The root-number character is also unchanged.  Indeed, by multiplicativity,
each paired constituent $\xi\oplus\xi^{c,-1}$ contributes trivially to the
root numbers occurring in ;
this is the identity used in \cite[(5.1)]{GGP12}.  Therefore, Conjecture 2 for
$(V,\phi)$ is reduced to the good-parity parameter
$(V_{\mathrm{gp}},\phi_{\mathrm{gp}})$.

The reduction for $(W,\phi')$ is the same.  Consequently, from now on we may
 assume that both $\phi$ and $\phi'$ have good parity.

\subsection{Endoscopic data used in the reductions}\label{subsec:endoscopic-data-reductions}

We record the endoscopic data used in the four reductions.  For $s_V\in A_\phi$, choose a lift acting by $+1$ and $-1$ on
\[
  M_V=M_{V,+}\oplus M_{V,-}.
\]
Equivalently, in the notation used in the calculation below, one may write
$s_V=c(I_{M_{V,+}},-I_{M_{V,-}})$.  Put
$n_\pm=\dim M_{V,\pm}$, so $n=n_++n_-$.  The corresponding elliptic
endoscopic group is
\[
  U(V_+)\times U(V_-),\qquad \dim V_\pm=n_\pm.
\]
The same construction for $s_W\in A_{\phi'}$ gives
$U(W_+)\times U(W_-)$.

We now spell out the $L$-embedding on the $V$-side.  Following
\cite[(2.1.12)--(2.1.13)]{Mok15}, set
\[
  \kappa_+=(-1)^{n_-},\qquad \kappa_-=(-1)^{n_+},
\]
and choose conjugate self-dual characters $\mu_+$ and $\mu_-$ of
$W_{\mathbb C}=\mathbb C^\times$ of signs $\kappa_+$ and $\kappa_-$,
respectively.  Thus, under local class field theory,
$\mu_\pm|_{\mathbb R^\times}$ is trivial when $\kappa_\pm=1$ and is
$\omega_{\mathbb C/\mathbb R}$ when $\kappa_\pm=-1$. For $d\geq1$, let
\[
  (\Phi_d)_{a,b}=(-1)^{a-1}\delta_{a,d+1-b}
  \qquad (1\leq a,b\leq d),
\]
and fix $j\in W_{\mathbb R}\setminus W_{\mathbb C}$.  Then the endoscopic
$L$-embedding
\[
  \zeta_{\mu_+,\mu_-}:{}^L\!\bigl(U(V_+)\times U(V_-)\bigr)
  \longrightarrow {}^LU(V)
\]
is represented by
\begin{align}
 (g_+,g_-)\rtimes1
 &\longmapsto
   \operatorname{diag}(g_+,g_-)\rtimes1,
                                                        \label{eq:endoscopic-embedding-neutral}\\
 (I_{n_+},I_{n_-})\rtimes z
 &\longmapsto
   \operatorname{diag}\bigl(\mu_+(z)I_{n_+},\mu_-(z)I_{n_-}\bigr)\rtimes z,
   \qquad z\in W_{\mathbb C},                         \label{eq:endoscopic-embedding-WC}\\
 (I_{n_+},I_{n_-})\rtimes j
 &\longmapsto
   \operatorname{diag}\bigl(\kappa_+\Phi_{n_+},
                             \kappa_-\Phi_{n_-}\bigr)
   \Phi_n^{-1}\rtimes j.                              \label{eq:endoscopic-embedding-j}
\end{align}
Its $\widehat{U(V)}$-conjugacy class is independent of the choice of $j$.
This is exactly the explicit realization of the endoscopic transfer that we
use below.

At the level of parameters the construction is summarized by the commutative
diagram
\begin{equation}\label{diag:endoscopic-parameter-V}
\begin{tikzcd}[column sep=large,row sep=large]
W_{\mathbb R}
  \arrow{r}{\phi}
  \arrow[swap]{dr}{(\phi_+,\phi_-)}
& {}^LU(V)
  \arrow{r}{\mathrm{std}}
& {}^L\!\operatorname{Res}_{\mathbb C/\mathbb R}\mathrm{GL}(V)\\
& {}^L\!\bigl(U(V_+)\times U(V_-)\bigr)
  \arrow[swap]{u}{\zeta_{\mu_+,\mu_-}}
&
\end{tikzcd}
\end{equation}
so that, after restriction to $W_{\mathbb C}$ and passage to the standard
representation,
\begin{equation}\label{eq:unitary-endoscopic-parameter-decomp}
  \phi=\mu_+\phi_+\oplus\mu_-\phi_-.
\end{equation}
There is an identical diagram on the $W$-side, with characters
$\mu'_+,\mu'_-$ and
\[
  \phi'=\mu'_+\phi'_+\oplus\mu'_-\phi'_-.
\]
The compatibility of \eqref{diag:endoscopic-parameter-V} with standard base
change is the commutative diagram of \cite[(2.1.14)]{Mok15}.

The virtual packets and their endoscopic character transfer were fixed in
Subsection~\ref{subsec:endoscopic-character-transfer}.  In the arguments below
we apply the geometric endoscopic reduction formula
Theorem~\ref{thm:geometric-endoscopy-reduction}, principally through its summed
form \eqref{eq:endoscopic-reduction-summed}.

\subsection{The reduction method from the endoscopic reduction formula}

\subsubsection{The second reduction}
\label{subsec:second-reduction}

After the parabolic reduction, no bad-parity summand will occur in the
remaining argument.  Thus, from now on, we write
\begin{equation}\label{eq:good-parity-two-parameters}
  \phi=\bigoplus_{i=1}^{M}m_i\chi_i,
  \qquad
  \phi'=\bigoplus_{j=1}^{N}n_j\chi'_j,
\end{equation}
where the $\chi_i$ (respectively $\chi'_j$) are pairwise distinct conjugate
self-dual characters of the sign required by $U(V)$ (respectively $U(W)$).
In particular,
\[
  A_\phi=\bigoplus_{i=1}^{M}(\mathbb Z/2\mathbb Z)a_i,
  \qquad
  A_{\phi'}=\bigoplus_{j=1}^{N}(\mathbb Z/2\mathbb Z)b_j
\]
in the notation of \cite[\S2.1]{Xue23}.

\begin{thm}[Second reduction]\label{thm:second-reduction}
Assume that Conjecture~2 is known for all relevant Gross--Prasad pairs of
strictly smaller rank.  Then, for the good-parity parameters
\eqref{eq:good-parity-two-parameters}, it is enough to prove Conjecture~2 under
the additional assumption
\[
  m_i\leq 2\quad(1\leq i\leq M),
  \qquad
  n_j\leq 2\quad(1\leq j\leq N).
\]
More precisely, if $m_i\geq3$ (or $n_j\geq3$), the conjecture is reduced to the
parameter obtained by lowering this multiplicity by two, together with
Gross--Prasad pairs of smaller rank.
\end{thm}

\begin{proof}
We give the argument on the $V$-side; the $W$-side is identical. Suppose
$m_i\geq3$, $s_V\times s_W\in A_\phi\times A_{\phi'}$ satisfies the condition $dim(V^+)-dim(W^-)$ is even and $dim(V^-)-dim(W^+)$ is odd. For $s_W\ne1\in A_{\phi'}$ ,choose a lift of $s_V$ such that the summand
$(m_i-1)\chi_i$ lies in the minus endoscopic factor.  Thus
\begin{equation}\label{eq:old-endoscopic-decomposition}
  \phi=\mu_+\phi_+\oplus\mu_-\phi_-,
  \qquad (m_i-1)\chi_i\subset\mu_-\phi_-.
\end{equation}
Define
\begin{equation}\label{eq:reduced-multiplicity-parameter}
  \phi^{(i)}
   :=(m_i-2)\chi_i\oplus
     \bigoplus_{j\neq i}m_j\chi_j.
\end{equation}
It is a good-parity parameter of dimension $\dim V-2$.  Let $V^{(i)}$ be a
Hermitian space of this dimension.  We keep the plus endoscopic factor
unchanged and replace $V_-$ by $V_-^{(i)}$:
\begin{equation}\label{eq:dimension-reduced-endoscopy}
  \dim V^{(i)}=\dim V-2,\qquad
  \dim V_-^{(i)}=\dim V_--2,\qquad
  V_+^{(i)}=V_+.
\end{equation}
At the level of parameters,
\begin{equation}\label{eq:new-endoscopic-decomposition}
  \phi^{(i)}=\mu_+\phi_+\oplus\mu_-\phi_-^{(i)},
  \qquad
  \mu_-\phi_-=(\mu_-\phi_-^{(i)})\oplus2\chi_i.
\end{equation}

Since the two relevant dimensions have both changed by two, the parity
conditions defining the endoscopic $L$-embedding are unchanged.  Hence the
same pair $(\mu_+,\mu_-)$ can be used for $V^{(i)}$.  The corresponding
commutative diagram is
\begin{equation}\label{diag:endoscopic-parameter-V-reduced}
\begin{tikzcd}[column sep=large,row sep=large]
W_{\mathbb R}
  \arrow{r}{\phi^{(i)}}
  \arrow[swap]{dr}{(\phi_+,\phi_-^{(i)})}
& {}^LU(V^{(i)})
  \arrow{r}{\mathrm{std}}
& {}^L\!\operatorname{Res}_{\mathbb C/\mathbb R}\mathrm{GL}(V^{(i)})\\
& {}^L\!\bigl(U(V_+)\times U(V_-^{(i)})\bigr)
  \arrow[swap]{u}{\zeta_{\mu_+,\mu_-}}
&
\end{tikzcd}
\end{equation}
and is obtained from \eqref{diag:endoscopic-parameter-V} by lowering the
dimension of the minus block by two.

Because $m_i-2>0$, the corresponding generator is still present, and there is
a natural identification
\begin{equation}\label{eq:component-group-multiplicity-reduction}
  A_{\phi^{(i)}}\simeq A_\phi.
\end{equation}
Applying \eqref{eq:endoscopic-reduction-summed} to $(\phi,\phi')$ and then to
$(\phi^{(i)},\phi')$ gives
\begin{align}
 &\sum_{\alpha\in H^1(\mathbb R,U(W))}
 m\!\left({\Sigma}^{s_W}_{W_\alpha,\phi'},
          {\Sigma}^{s_V}_{V_\alpha,\phi}\right)
 =\gamma^V_{\mu_+,\mu_-}\gamma^W_{\mu_+',\mu_-'}S_{\mu_+\mu'_-}({\Sigma}_{V_+},{\Sigma}_{W_-})
  S_{\mu_-\mu'_+}({\Sigma}_{V_-},{\Sigma}_{W_+}),
                                                        \label{eq:old-reduction-product}\\
 &\sum_{\alpha\in H^1(\mathbb R,U(W))}
 m\!\left({\Sigma}^{s_W}_{W_\alpha,\phi'},
          {\Sigma}^{s_V}_{V^{(i)}_\alpha,\phi^{(i)}}\right)
 =\gamma^{V^{(i)}}_{\mu_+,\mu_-}\gamma^W_{\mu_+',\mu_-'}S_{\mu_+\mu'_-}({\Sigma}_{V_+},{\Sigma}_{W_-})
  S_{\mu_-\mu'_+}({\Sigma}_{V_-^{(i)}},{\Sigma}_{W_+}).
                                                        \label{eq:new-reduction-product}
\end{align}
Consequently,
\begin{align}
 &\sum_{\alpha\in H^1(\mathbb R,U(W))}
 m\!\left({\Sigma}^{s_W}_{W_\alpha,\phi'},
          {\Sigma}^{s_V}_{V_\alpha,\phi}\right)
 =
 \frac{
  \displaystyle
  \sum_{\alpha\in H^1(\mathbb R,U(W))}
  m\!\left({\Sigma}^{s_W}_{W_\alpha,\phi'},
            {\Sigma}^{s_V}_{V^{(i)}_\alpha,\phi^{(i)}}\right)
 }{
  S_{\mu_-\mu'_+}({\Sigma}_{V_-^{(i)}},{\Sigma}_{W_+})
 }
 S_{\mu_-\mu'_+}({\Sigma}_{V_-},{\Sigma}_{W_+}).
                                                        \label{eq:multiplicity-reduction-ratio}
\end{align}
The two stable terms on the last line are attached to smaller Gross--Prasad
pairs, hence are known by induction.

It remains to compare the character prescribed by the epsilon dichotomy.
For every parameter $\rho$, multiplicativity of local root numbers gives
\begin{equation}\label{eq:epsilon-remove-two-copies}
  \frac{\epsilon(\frac12,\mu_-\phi_-\otimes\rho,\psi_{\mathbb C})}
       {\epsilon(\frac12,\mu_-\phi_-^{(i)}\otimes\rho,\psi_{\mathbb C})}
   =\epsilon(\tfrac12,\chi_i\otimes\rho,\psi_{\mathbb C})^2=1.
\end{equation}
Thus deleting two copies of $\chi_i$ does not change the Xue character under
\eqref{eq:component-group-multiplicity-reduction}.  Combining
\eqref{eq:multiplicity-reduction-ratio} and
\eqref{eq:epsilon-remove-two-copies} proves the reduction, and the multiplication of some character values on $s_W= -1\in A_{\phi'}$ could get the value on $s_W=1$. Repeating the
argument on both sides yields $m_i,n_j\leq2$ for all $i,j$.
\end{proof}

\begin{rem}
    Actually we could use the same method in \eqref{thm:second-reduction} to prove the first reduction. The explicit construction method is more explicit before, and this gives a new viewpoint about the "induced representation" under the endoscopic theory.
\end{rem}

\subsubsection{The third reduction}
\label{subsec:third-reduction}

We now assume the conclusion of Theorem~\ref{thm:second-reduction}.  Thus the
parameters are still of good parity and every multiplicity is at most two.
The next step reduces the size of the component groups.  The construction is
the unitary analogue of the endoscopic descent used in \cite[\S5]{CL22}.

For later use, if $s_V\in A_\phi$ and $s_W\in A_{\phi'}$, choose lifts and
write the corresponding decompositions as
\[
  \phi=\phi_{V,+}\oplus\phi_{V,-},
  \qquad
  \phi'=\phi'_{W,+}\oplus\phi'_{W,-}.
\]
Here the notation includes the endoscopic twisting characters; equivalently,
if
\[
  \phi=\mu_+\phi_+\oplus\mu_-\phi_-,
  \qquad
  \phi'=\mu'_+\phi'_+\oplus\mu'_-\phi'_-,
\]
then
\[
  \phi_{V,+}=\mu_+\phi_+,\quad
  \phi_{V,-}=\mu_-\phi_-,
  \qquad
  \phi'_{W,+}=\mu'_+\phi'_+,\quad
  \phi'_{W,-}=\mu'_-\phi'_-.
\]
We choose the labels so that
\begin{equation}\label{eq:cross-parity-third-reduction}
  \dim V_+-\dim W_-\equiv0\pmod2,
  \qquad
  \dim V_- -\dim W_+\equiv1\pmod2.
\end{equation}
With this convention the value of the character prescribed by
\cite[Theorem~2.2]{Xue23} can be written
\begin{equation}\label{eq:xue-character-endoscopic-form}
 \eta_{\phi\times\phi'}(s_V,s_W)
  =
  \epsilon\!\left(\tfrac12,\phi_{V,-}\otimes\phi',\psi_{\mathbb C}\right)
  \epsilon\!\left(\tfrac12,\phi\otimes\phi'_{W,-},\psi_{\mathbb C}\right).
\end{equation}
By multiplicativity, this is equivalently
\begin{equation}\label{eq:xue-character-cross-form}
 \eta_{\phi\times\phi'}(s_V,s_W)
 =
 \epsilon\!\left(\tfrac12,\phi_{V,-}\otimes\phi'_{W,+},\psi_{\mathbb C}\right)
 \epsilon\!\left(\tfrac12,\phi_{V,+}\otimes\phi'_{W,-},\psi_{\mathbb C}\right),
\end{equation}
because the factor
$\epsilon(\frac12,\phi_{V,-}\otimes\phi'_{W,-},\psi_{\mathbb C})$
occurs twice.

\begin{thm}[Third reduction: pointwise endoscopic descent]\label{thm:third-reduction}
Assume that Conjecture~2 is known for all relevant Gross--Prasad pairs of
strictly smaller rank, and assume that Theorem~\ref{thm:second-reduction} has
already been applied.  Fix an element
\[
  (s_V,s_W)\in A_\phi\times A_{\phi'}
\]
and let
\[
  U(V_+)\times U(V_-),\qquad U(W_+)\times U(W_-)
\]
be the corresponding endoscopic groups.  Choose the labels so that
\begin{equation}\label{eq:third-reduction-parity-condition}
  \dim V_+-\dim W_-\equiv0\pmod2,
  \qquad
  \dim V_- -\dim W_+\equiv1\pmod2.
\end{equation}
Assume moreover that
\begin{equation}\label{eq:third-reduction-nontriviality-condition}
  V_-\neq0,\qquad W_+\neq0,
  \qquad\text{and}\qquad
  \bigl(V_-\neq V\ \text{or}\ W_+\neq W\bigr).
\end{equation}
Then the epsilon-dichotomy identity attached to the fixed element
$(s_V,s_W)$ is reduced, through the endoscopic reduction formula, to
identities for relevant Gross--Prasad pairs of strictly smaller rank.
\end{thm}

\begin{proof}
The proof is the descent constructed in the handwritten argument.  Since
$V_-\neq0$ and $W_+\neq0$, choose constituents
\[
  \chi_1\subset\phi_{V,-},
  \qquad
  \chi'_1\subset\phi'_{W,+}.
\]
After reordering \eqref{eq:good-parity-two-parameters}, write these as the
first constituents.  Choose a unitary character $\xi$ of $\mathbb C^\times$
such that
\begin{equation}\label{eq:xi-sign-character}
  \xi|_{\mathbb R^\times}=\operatorname{sgn}.
\end{equation}
Define parameters of dimensions one less by
\begin{align}
 \phi^{\mathrm{new}}
  &=
  \left((m_1-1)\chi_1
       \oplus\bigoplus_{i=2}^{M}m_i\chi_i\right)\otimes\xi,
                                                        \label{eq:third-new-phi}\\
 {\phi'}^{\mathrm{new}}
  &=
  \left((n_1-1)\chi'_1
       \oplus\bigoplus_{j=2}^{N}n_j\chi'_j\right)\otimes\xi^{-1}.
                                                        \label{eq:third-new-phiprime}
\end{align}
The twist by $\xi$ changes the conjugate-self-dual sign, exactly compensating
for the change of dimension.  Hence these are parameters for Hermitian spaces
$V^{\mathrm{new}}$ and $W^{\mathrm{new}}$ with
\[
  \dim V^{\mathrm{new}}=\dim V-1,
  \qquad
  \dim W^{\mathrm{new}}=\dim W-1.
\]

The construction is compatible with the original endoscopic factorizations.
On the $V$-side, the old factorization is
\begin{equation}\label{diag:third-old-V}
\begin{tikzcd}[column sep=large,row sep=large]
W_{\mathbb R}
  \arrow{r}{\phi}
  \arrow[swap]{dr}{(\phi_+,\phi_-)}
& {}^LU(V)
  \arrow{r}{\mathrm{std}}
& {}^L\!\operatorname{Res}_{\mathbb C/\mathbb R}\mathrm{GL}(V)\\
& {}^L\!\bigl(U(V_+)\times U(V_-)\bigr)
  \arrow[swap]{u}{\zeta_{\mu_+,\mu_-}}
&
\end{tikzcd}
\end{equation}
whereas after deleting $\chi_1$ from the minus block it becomes
\begin{equation}\label{diag:third-new-V}
\begin{tikzcd}[column sep=large,row sep=large]
W_{\mathbb R}
  \arrow{r}{\phi^{\mathrm{new}}}
  \arrow[swap]{dr}{(\phi_+,\phi_-^{\mathrm{new}})}
& {}^LU(V^{\mathrm{new}})
  \arrow{r}{\mathrm{std}}
& {}^L\!\operatorname{Res}_{\mathbb C/\mathbb R}
       \mathrm{GL}(V^{\mathrm{new}})\\
& {}^L\!\bigl(U(V_+)\times U(V_-^{\mathrm{new}})\bigr)
  \arrow[swap]{u}{\zeta_{\mu_+\xi,\mu_-}}
&
\end{tikzcd}
\end{equation}
with $\dim V_-^{\mathrm{new}}=\dim V_--1$.  Only the character attached to the
plus block changes; this is the parity change in the endoscopic embedding when
the dimension of the minus block drops by one.

Likewise, on the $W$-side,
\begin{equation}\label{diag:third-old-W}
\begin{tikzcd}[column sep=large,row sep=large]
W_{\mathbb R}
  \arrow{r}{\phi'}
  \arrow[swap]{dr}{(\phi'_+,\phi'_-)}
& {}^LU(W)
  \arrow{r}{\mathrm{std}}
& {}^L\!\operatorname{Res}_{\mathbb C/\mathbb R}\mathrm{GL}(W)\\
& {}^L\!\bigl(U(W_+)\times U(W_-)\bigr)
  \arrow[swap]{u}{\zeta_{\mu'_+,\mu'_-}}
&
\end{tikzcd}
\end{equation}
is replaced by
\begin{equation}\label{diag:third-new-W}
\begin{tikzcd}[column sep=large,row sep=large]
W_{\mathbb R}
  \arrow{r}{{\phi'}^{\mathrm{new}}}
  \arrow[swap]{dr}{({\phi'_+}^{\mathrm{new}},\phi'_-)}
& {}^LU(W^{\mathrm{new}})
  \arrow{r}{\mathrm{std}}
& {}^L\!\operatorname{Res}_{\mathbb C/\mathbb R}
       \mathrm{GL}(W^{\mathrm{new}})\\
& {}^L\!\bigl(U(W_+^{\mathrm{new}})\times U(W_-)\bigr)
  \arrow[swap]{u}{\zeta_{\mu'_+,\mu'_-\xi^{-1}}}
&
\end{tikzcd}
\end{equation}
with $\dim W_+^{\mathrm{new}}=\dim W_+-1$.  Notice that
\begin{equation}\label{eq:third-twist-cancellation}
  (\mu_+\xi)(\mu'_-\xi^{-1})=\mu_+\mu'_-,
  \qquad
  \mu_-\mu'_+=\mu_-\mu'_+.
\end{equation}
Thus the twisting character in each cross Gross--Prasad factor is unchanged
by the descent.

For brevity put
\begin{equation}\label{eq:calM-definition}
 \mathcal M_{\phi,\phi'}(s_V,s_W)
 :=
 \sum_{\alpha\in H^1(\mathbb R,U(W))}
 m\!\left({\Sigma}^{s_W}_{W_\alpha,\phi'},
          {\Sigma}^{s_V}_{V_\alpha,\phi}\right).
\end{equation}
The summed endoscopic reduction formula gives
\begin{align}
 \mathcal M_{\phi,\phi'}(s_V,s_W)
 &=\gamma^V_{\mu_+,\mu_-}\gamma^W_{\mu_+',\mu_-'}
 S_{\mu_+\mu'_-}({\Sigma}_{V_+},{\Sigma}_{W_-})
 S_{\mu_-\mu'_+}({\Sigma}_{V_-},{\Sigma}_{W_+}),
                                                        \label{eq:third-old-product}\\
 \mathcal M_{\phi^{\mathrm{new}},{\phi'}^{\mathrm{new}}}
   (s_V^{\mathrm{new}},s_W^{\mathrm{new}})
 &=\gamma^{V^{\mathrm{new}}}_{\mu_+\xi,\mu_-}\gamma^{W^{\mathrm{new}}}_{\mu_+',\mu_-'\xi^{-1}}
 S_{\mu_+\mu'_-}({\Sigma}_{V_+},{\Sigma}_{W_-})
 S_{\mu_-\mu'_+}
   ({\Sigma}_{V_-^{\mathrm{new}}},
    {\Sigma}_{W_+^{\mathrm{new}}}).
                                                        \label{eq:third-new-product}
\end{align}
Here $s_V^{\mathrm{new}}$ and $s_W^{\mathrm{new}}$ are the elements induced by
the displayed endoscopic factorizations.  Dividing the two identities gives
\begin{align}
 \mathcal M_{\phi,\phi'}(s_V,s_W)
 &=
 \frac{
 \mathcal M_{\phi^{\mathrm{new}},{\phi'}^{\mathrm{new}}}
 (s_V^{\mathrm{new}},s_W^{\mathrm{new}})
 }{
 S_{\mu_-\mu'_+}
  ({\Sigma}_{V_-^{\mathrm{new}}},
   {\Sigma}_{W_+^{\mathrm{new}}})
 }
 S_{\mu_-\mu'_+}({\Sigma}_{V_-},{\Sigma}_{W_+}).
                                                        \label{eq:third-ratio}
\end{align}
Condition \eqref{eq:third-reduction-nontriviality-condition} is precisely what
ensures that the cross pair $(W_+,V_-)$ occurring here is not the original
pair, while the new pair has both dimensions lowered by one.  Hence every term
on the right is attached to a strictly smaller-rank Gross--Prasad problem (the
other cross factor may be $(0,0)$, in which case we use
\eqref{eq:S00}).

We finally compare the character prescribed by the epsilon dichotomy.  By the
induction hypothesis, the right-hand side of \eqref{eq:third-ratio} is
evaluated by the corresponding root numbers.  Using
\eqref{eq:xue-character-endoscopic-form} for the new pair and the cancelation
of the twists $\xi$ and $\xi^{-1}$, it becomes
\begin{align}
 &\eta_{\phi^{\mathrm{new}}\times{\phi'}^{\mathrm{new}}}
      (s_V^{\mathrm{new}},s_W^{\mathrm{new}})
 \frac{
  \epsilon(\frac12,\phi_{V,-}\otimes\phi'_{W,+},\psi_{\mathbb C})
 }{
  \epsilon(\frac12,
    (\phi_{V,-}-\chi_1)\otimes(\phi'_{W,+}-\chi'_1),
    \psi_{\mathbb C})
 }                                                     \notag\\
 &\quad=
 \frac{
  \epsilon(\frac12,
     (\phi_{V,-}-\chi_1)\otimes(\phi'-\chi'_1),
     \psi_{\mathbb C})
  \epsilon(\frac12,
     (\phi-\chi_1)\otimes\phi'_{W,-},
     \psi_{\mathbb C})
 }{
  \epsilon(\frac12,
     (\phi_{V,-}-\chi_1)\otimes(\phi'_{W,+}-\chi'_1),
     \psi_{\mathbb C})
 }
 \epsilon(\tfrac12,\phi_{V,-}\otimes\phi'_{W,+},\psi_{\mathbb C}) \notag\\
 &\quad=
 \epsilon(\tfrac12,\phi_{V,+}\otimes\phi'_{W,-},\psi_{\mathbb C})
 \epsilon(\tfrac12,\phi_{V,-}\otimes\phi'_{W,+},\psi_{\mathbb C}) \notag\\
 &\quad=
 \epsilon(\tfrac12,\phi\otimes\phi'_{W,-},\psi_{\mathbb C})
 \epsilon(\tfrac12,\phi_{V,-}\otimes\phi',\psi_{\mathbb C})
 =\eta_{\phi\times\phi'}(s_V,s_W).
                                                        \label{eq:third-epsilon-computation}
\end{align}
This proves the desired descent for the fixed component-group element
$(s_V,s_W)$.
\end{proof}

\subsubsection{The fourth reduction}
\label{subsec:fourth-reduction}

The preceding reduction removes one constituent simultaneously from the two
crossed endoscopic blocks.  We shall also use a second, slightly different,
reduction in which two distinct constituents are removed from the same
endoscopic block.  This is the fourth reduction in the argument.

\begin{thm}[Fourth reduction]\label{thm:fourth-reduction}
Assume that Conjecture~2 is known for all relevant Gross--Prasad pairs of
strictly smaller rank.  Let
\[
  \phi=m_1\phi_1\oplus m_2\phi_2\oplus
       \bigoplus_{i=3}^{M}m_i\phi_i,
  \qquad
  \phi'=\bigoplus_{j=1}^{N}n_j\phi'_j
\]
be good-parity tempered parameters, and fix
$(s_V,s_W)\in A_\phi\times A_{\phi'}$.  Choose the labels of the associated
endoscopic factors so that
\[
  \dim V_+-\dim W_-\equiv0\pmod2,
  \qquad
  \dim V_- -\dim W_+\equiv1\pmod2,
\]
\[
  V_-\ne V \qquad \text{or} \qquad W^+\neq W
\]
Suppose that $\phi_1$ and $\phi_2$ are two distinct irreducible constituents
of the minus block $\phi_{V,-}$.  Then the epsilon-dichotomy identity at
$(s_V,s_W)$ reduces to identities for strictly smaller-rank Gross--Prasad
pairs.
\end{thm}

\begin{proof}
Remove one copy of each of $\phi_1$ and $\phi_2$ and put
\begin{equation}\label{eq:fourth-new-parameter}
  \phi^{\dagger}
  :=(m_1-1)\phi_1\oplus(m_2-1)\phi_2
     \oplus\bigoplus_{i=3}^{M}m_i\phi_i.
\end{equation}
Since the dimension drops by two, the good-parity sign is unchanged; hence
$\phi^{\dagger}$ is a tempered parameter for a Hermitian space
$V^{\dagger}$ with
\[
  \dim V^{\dagger}=\dim V-2.
\]
The two removed constituents lie in the minus block, so the corresponding
endoscopic decomposition satisfies
\[
  \dim V_+^{\dagger}=\dim V_+,
  \qquad
  \dim V_-^{\dagger}=\dim V_- -2.
\]
In particular, the parity of both endoscopic dimensions is unchanged, and the
same pair of endoscopic characters $(\mu_+,\mu_-)$ can be used. Thus, the old
and new endoscopic factorizations fit into the diagrams
\begin{equation}\label{diag:fourth-old-V}
\begin{tikzcd}[column sep=large,row sep=large]
W_{\mathbb R}
  \arrow{r}{\phi}
  \arrow[swap]{dr}{(\phi_+,\phi_-)}
& {}^LU(V)
  \arrow{r}{\mathrm{std}}
& {}^L\!\operatorname{Res}_{\mathbb C/\mathbb R}\mathrm{GL}(V)\\
& {}^L\!\bigl(U(V_+)\times U(V_-)\bigr)
  \arrow[swap]{u}{\zeta_{\mu_+,\mu_-}}
&
\end{tikzcd}
\end{equation}
and
\begin{equation}\label{diag:fourth-new-V}
\begin{tikzcd}[column sep=large,row sep=large]
W_{\mathbb R}
  \arrow{r}{\phi^{\dagger}}
  \arrow[swap]{dr}{(\phi_+,\phi_-^{\dagger})}
& {}^LU(V^{\dagger})
  \arrow{r}{\mathrm{std}}
& {}^L\!\operatorname{Res}_{\mathbb C/\mathbb R}\mathrm{GL}(V^{\dagger})\\
& {}^L\!\bigl(U(V_+)\times U(V_-^{\dagger})\bigr)
  \arrow[swap]{u}{\zeta_{\mu_+,\mu_-}}
&
\end{tikzcd}.
\end{equation}
This is simply the restriction of the Langlands correspondence compatible
with the original endoscopic datum.

The geometric endoscopic reduction formula gives
\begin{align}
 \mathcal M_{\phi,\phi'}(s_V,s_W)
 &=\gamma^V_{\mu_+,\mu_-}\gamma^W_{\mu_+',\mu_-'}S_{\mu_+\mu'_-}({\Sigma}_{V_+},{\Sigma}_{W_-})
   S_{\mu_-\mu'_+}({\Sigma}_{V_-},{\Sigma}_{W_+}),
                                                        \label{eq:fourth-old-product}\\
 \mathcal M_{\phi^{\dagger},\phi'}(s_V^{\dagger},s_W)
 &=\gamma^{V^\dagger}_{\mu_+,\mu_-}\gamma^W_{\mu_+',\mu_-'}S_{\mu_+\mu'_-}({\Sigma}_{V_+},{\Sigma}_{W_-})
   S_{\mu_-\mu'_+}({\Sigma}_{V_-^{\dagger}},{\Sigma}_{W_+}).
                                                        \label{eq:fourth-new-product}
\end{align}
Consequently,
\begin{equation}\label{eq:fourth-ratio}
 \mathcal M_{\phi,\phi'}(s_V,s_W)
 =
 \frac{\mathcal M_{\phi^{\dagger},\phi'}(s_V^{\dagger},s_W)}
 {S_{\mu_-\mu'_+}({\Sigma}_{V_-^{\dagger}},{\Sigma}_{W_+})}
 S_{\mu_-\mu'_+}({\Sigma}_{V_-},{\Sigma}_{W_+}).
\end{equation}
All terms on the right are attached to smaller-rank Gross--Prasad pairs.
Applying the induction hypothesis to \eqref{eq:fourth-ratio}, the right-hand
side becomes
\begin{align}
&\epsilon\!\left(\tfrac12,\phi_{V,-}\otimes\phi'_{W,+},\psi_{\mathbb C}\right)
 \frac{
 \epsilon\!\left(\tfrac12,\phi_{V,-}^{\dagger}\otimes\phi',\psi_{\mathbb C}\right)
 \epsilon\!\left(\tfrac12,\phi^{\dagger}\otimes\phi'_{W,-},\psi_{\mathbb C}\right)
 }{
 \epsilon\!\left(\tfrac12,\phi_{V,-}^{\dagger}\otimes\phi'_{W,+},\psi_{\mathbb C}\right)
 }                                                     \notag\\
&\qquad=
 \epsilon\!\left(\tfrac12,\phi_{V,-}\otimes\phi'_{W,+},\psi_{\mathbb C}\right)
 \epsilon\!\left(\tfrac12,\phi_{V,+}\otimes\phi'_{W,-},\psi_{\mathbb C}\right)
                                                        \notag\\
&\qquad=
 \epsilon\!\left(\tfrac12,\phi_{V,-}\otimes\phi',\psi_{\mathbb C}\right)
 \epsilon\!\left(\tfrac12,\phi\otimes\phi'_{W,-},\psi_{\mathbb C}\right)
 =\eta_{\phi\times\phi'}(s_V,s_W).
                                                        \label{eq:fourth-epsilon-computation}
\end{align}
Here we used only multiplicativity of epsilon factors; the two occurrences of
$\epsilon(\frac12,\phi_{V,-}^{\dagger}\otimes\phi'_{W,-},\psi_{\mathbb C})$
cancel.  This proves the fourth reduction.
\end{proof}

\subsection{Reduction to the basic cases}

\subsubsection{Reduction of the component group size}
\begin{lem}[Reduction of the component groups]\label{lem:component-group-reduction}
Assume the hypotheses of Theorems~\ref{thm:third-reduction} and
\ref{thm:fourth-reduction}.  If
\begin{equation}\label{eq:component-group-large-condition}
  |A_\phi|\geq8
  \qquad\text{or}\qquad
  |A_{\phi'}|\geq8,
\end{equation}
then the epsilon-dichotomy identities for every element of
$A_\phi\times A_{\phi'}$ reduce to epsilon-dichotomy identities for
strictly smaller-rank Gross--Prasad pairs.
\end{lem}

\begin{proof}
By symmetry, assume $|A_\phi|\geq8$.  After the good-parity and multiplicity
reductions we may write
\[
 A_\phi\simeq\{\pm1\}^{M},\qquad M\geq3,
 \qquad
 A_{\phi'}\simeq\{\pm1\}^{N},
\]
using the coordinates associated with the distinct good-parity constituents,
as in \cite[\S2.1]{Xue23}.  We write
$\mathbf 1=(1,\ldots,1)$ and $-\mathbf1=(-1,\ldots,-1)$.  Since
$\eta_{\phi\times\phi'}$ is a character, once its value is known on a set of
generators of $A_\phi\times A_{\phi'}$, all remaining values are determined.
We distinguish the three cases appearing in the handwritten argument.

\smallskip
\noindent\emph{Case 1: $|A_{\phi'}|\geq4$.}
For
\[
 t\in A_\phi\setminus\{\mathbf1,-\mathbf1\},
 \qquad
 s\in A_{\phi'}\setminus\{\mathbf1,-\mathbf1\},
\]
both the plus and minus endoscopic blocks are nonzero.  After interchanging
the two labels of an endoscopic group when necessary, the parity convention
of Theorem~\ref{thm:third-reduction} is satisfied, and that theorem determines
$\eta_{\phi\times\phi'}(t,s)$.

Choose
\[
 e_1=(-1,1,\ldots,1),\qquad e_2=(1,-1,1,\ldots,1)
\]
and fix any non-extreme $s\in A_{\phi'}$.  The values at
$(e_1,s)$, $(e_2,s)$ and $(e_1e_2,s)$ are therefore known.  Multiplicativity
gives, for example,
\[
 \eta_{\phi\times\phi'}(e_1,\mathbf1)
 =\eta_{\phi\times\phi'}(e_1e_2,s)
  \eta_{\phi\times\phi'}(e_2,s),
 \qquad
 \eta_{\phi\times\phi'}(\mathbf1,s)
 =\eta_{\phi\times\phi'}(e_1,s)
  \eta_{\phi\times\phi'}(e_1,\mathbf1).
\]
The same calculation with the coordinate generators on both sides determines
the character on a generating set of $A_\phi\times A_{\phi'}$.  Hence all
values are determined by smaller-rank problems.

\smallskip
\noindent\emph{Case 2: $|A_{\phi'}|=2$ and $\phi'=2\phi'_1$.}
Let $b$ be the nontrivial element of $A_{\phi'}$.  Since the multiplicity of
$\phi'_1$ is two, we choose a lift of $b$ whose $+1$- and $-1$-eigenspaces in
the multiplicity space are both one-dimensional.  Thus both $W_{b,+}$ and
$W_{b,-}$ are nonzero.  Theorem~\ref{thm:third-reduction} therefore determines
$\eta_{\phi\times\phi'}(t,b)$ for every non-extreme $t\in A_\phi$.  Using
$e_1,e_2$ as in Case~1 and multiplicativity, we recover the values on
$A_\phi\times\{\mathbf1,b\}$, hence on all of
$A_\phi\times A_{\phi'}$.

\smallskip
\noindent\emph{Case 3: $|A_{\phi'}|=2$ and $\phi'$ is irreducible.}
Let again $b$ denote the nontrivial element.  In this case $\dim W=1$.  First
consider an element $t\in A_\phi$ having an even number of $-1$ coordinates.
Then both $\dim V_{t,+}$ and $\dim V_{t,-}$ are even.  If $t$ is non-extreme,
we label the one-dimensional $W$-factor as $W_-$ and put $W_+=0$; the parity
conditions of Theorem~\ref{thm:third-reduction} are then satisfied, and that
theorem determines $\eta_{\phi\times\phi'}(t,b)$.

It remains to obtain one element outside this even-parity subgroup.  Since
$M\geq3$, take
\[
 t_0=(-1,-1,-1,1,\ldots,1).
\]
For this element the minus block contains at least three distinct
constituents.  Label the $W$-endoscopy by $W_+=0$, $W_-=W$.  Then the parity
conditions in Theorem~\ref{thm:fourth-reduction} hold, and two of the distinct
constituents in $V_{t_0,-}$ may be removed.  The fourth reduction therefore
determines $\eta_{\phi\times\phi'}(t_0,b)$ from smaller-rank pairs.

The even-parity subgroup of $A_\phi$ together with $t_0$ generates
$A_\phi$.  Moreover, choosing for instance
\[
 t_1=(-1,-1,1,\ldots,1),\qquad
 t_2=(-1,1,-1,1,\ldots,1),
\]
we know the three values at $(t_1,b)$, $(t_2,b)$ and $(t_1t_2,b)$, and hence
\[
 \eta_{\phi\times\phi'}(\mathbf1,b)
 =\eta_{\phi\times\phi'}(t_1,b)
  \eta_{\phi\times\phi'}(t_2,b)
  \eta_{\phi\times\phi'}(t_1t_2,b).
\]
Consequently the values with second coordinate $\mathbf1$ are also recovered
by multiplicativity.  Thus the epsilon-dichotomy identity is determined for
every element of $A_\phi\times A_{\phi'}$.

The case $|A_{\phi'}|\geq8$ is identical with the roles of $\phi$ and
$\phi'$ interchanged.  This proves the lemma.
\end{proof}

\subsubsection{The remaining small component-group cases}
\label{subsec:remaining-small-component-groups}

We now check the remaining cases after the four reductions.  By
Lemma~\ref{lem:component-group-reduction}, we may assume
\[
  |A_\phi|\leq 4,
  \qquad
  |A_{\phi'}|\leq 4,
\]
and by Theorem~\ref{thm:second-reduction} every irreducible good-parity
constituent occurs with multiplicity at most two.  Up to relabeling the
distinct constituents, the remaining nontrivial configurations are
\begin{align*}
 \text{(i)}\quad
   &\phi=\chi_1\oplus\chi_2\quad\text{or}\quad2\chi_1,
   &\phi'&=\chi'_1,\\
 \text{(ii)}\quad
   &\phi=\chi_1\oplus2\chi_2,
   &\phi'&=2\chi'_1,\\
 \text{(iii)}\quad
   &\phi=\chi_1\oplus2\chi_2,
   &\phi'&=\chi'_1\oplus\chi'_2,\\
 \text{(iv)}\quad
   &\phi=2\chi_1\oplus2\chi_2,
   &\phi'&=\chi'_1\oplus2\chi'_2,\\
 \text{(v)}\quad
   &\phi=2\chi_1\oplus 2\chi_2,
   &\phi'&=\chi_1'
\end{align*}
Here, distinct symbols in each parameter denote distinct characters.
Case~\textup{(i)} will be checked directly in Section~\ref{sec:basic-case}.  We show that the other four configurations reduce to it.

\begin{prop}[Reduction of the remaining small cases]
\label{prop:remaining-small-cases}
Assume Conjecture~2 is known in Case~\textup{(i)}.  Then the
epsilon-dichotomy identities in Cases~\textup{(ii)}--\textup{(v)} follow
from the third reduction and the fourth reduction together with the component-group argument used
in Lemma~\ref{lem:component-group-reduction}.
\end{prop}

\begin{proof}
We use the coordinates on $A_\phi$ and $A_{\phi'}$ associated with the
distinct good-parity constituents and choose representatives of the
component-group elements with the indicated dimensions of the endoscopic
blocks. We only supply the arguments for \textup{(ii)} and \textup{(iv)}, the other two cases are similar.

\smallskip
\noindent\emph{Case \textup{(ii)}.}
Here
\[
  \phi=\chi_1\oplus2\chi_2,
  \qquad
  \phi'=2\chi'_1.
\]
Consider first
\[
  s_1=((1,-1),-1)\in A_\phi\times A_{\phi'}.
\]
We choose the corresponding endoscopic decomposition so that
\[
  \dim V_+=2,\qquad \dim V_-=1,
  \qquad
  \dim W_-=1,\qquad \dim W_+=1.
\]
Thus the third reduction reduces the identity at $s_1$ to
Case~\textup{(i)}.  For
\[
  s_2=((-1,1),-1)
\]
we may choose the same dimension pattern,
\[
  \dim V_+=2,\qquad \dim V_-=1,
  \qquad
  \dim W_-=1,\qquad \dim W_+=1,
\]
and the third reduction again reduces the desired identity to
Case~\textup{(i)}.  Finally, for
\[
  s_3=((-1,-1),-1)
\]
we choose
\[
  \dim V_+=1,\qquad \dim V_-=2,
  \qquad
  \dim W_-=1,\qquad \dim W_+=1,
\]
so the same reduction applies.  Since
\[
  s_1s_2s_3=((1,1),-1),
\]
the value at the remaining element is determined by the same
component-group character argument as in the proof of
Lemma~\ref{lem:component-group-reduction}.  Hence Case~\textup{(ii)} is
reduced to Case~\textup{(i)}.

\smallskip
\noindent\emph{Case \textup{(iv)}.}
Now let
\[
  \phi=2\chi_1\oplus2\chi_2,
  \qquad
  \phi'=\chi'_1\oplus2\chi'_2.
\]
Choose
\[
  s_W=(1,-1)\qquad\text{or}\qquad s_W=(-1,1)
\]
so that the associated endoscopic decomposition of $W$ has
\[
  \dim W_-=1,
  \qquad
  \dim W_+=2.
\]
For
\[
  s_{1,V}=(1,-1)\in A_\phi
\]
choose a representative satisfying
\[
  \dim V_+=3,
  \qquad
  \dim V_-=1.
\]
The third reduction therefore reduces the identity at
$(s_{1,V},s_W)$ to Case~\textup{(i)}.  The same argument applies to
\[
  s_{2,V}=(-1,1),
\]
for which we again choose
\[
  \dim V_+=3,
  \qquad
  \dim V_-=1.
\]
For
\[
  s_{3,V}=(-1,-1)
\]
we choose instead
\[
  \dim V_+=2,
  \qquad
  \dim V_-=2,
\]
and the third reduction again reduces the corresponding identity to a
smaller GGP pair. Since
\[
  s_{1,V}s_{2,V}s_{3,V}=(1,1)\in A_\phi,
  \qquad
  s_W s_W s_W=s_W\in A_{\phi'}
\]
The remaining value is determined by the same argument used in the
proof of Lemma~\ref{lem:component-group-reduction}.

\end{proof}

\section{The basic cases}\label{sec:basic-case}

We now verify the basic cases in dimensions
$(\dim W,\dim V)=(1,2)$, using the epsilon-dichotomy
normalization fixed in Section~\ref{sec:reduction-basic-cases}.

For $k\in\mathbb{Z}$ and $t\in\mathbb{R}$, put
\[
  \chi_{k+it}(z)
  =\left(\frac{z}{\sqrt{z\bar z}}\right)^k
   (z\bar z)^{it/2},
  \qquad z\in\mathbb{C}^\times,
\]
and write $\chi_k=\chi_{k+0i}$.

There are two relevant pairs:
\begin{equation}\label{eq:two-relevant-pairs}
 (\mathrm{U}(1,0),\mathrm{U}(1,1)),\qquad
 (\mathrm{U}(0,1),\mathrm{U}(0,2)).
\end{equation}
We fix
\[
  \phi_1=\chi_{2m},\qquad m\in\mathbb{Z},
\]
whose corresponding representation on either one-dimensional unitary group
will again be denoted by $\chi_m$.

\smallskip
\noindent\emph{The discrete series.}
We first consider the discrete parameter
\begin{equation}\label{eq:basic-discrete-parameter}
 \phi_2=\chi_{\ell+n-1}\oplus\chi_{\ell-n+1},
 \qquad n\ge2,\quad \ell\equiv n\pmod2.
\end{equation}
Then $A_{\phi_2}\simeq(\mathbb Z/2\mathbb Z)^2$.  The two packet members on
$\mathrm{U}(1,1)$ are $D_n^+\boxtimes\chi_\ell$ and
$D_n^-\boxtimes\chi_\ell$, while the relevant member on $\mathrm{U}(0,2)$ is
\begin{equation}\label{eq:compact-member}
 F_{n,\ell}
 ={\operatorname{Sym}^{n-2}(\mathbb{C}^2)}
   \otimes \mathrm{det}^{\frac{-n+\ell}{2}+1}.
\end{equation}
These packet descriptions follow from the real local Langlands correspondence;
see \cite[Theorem~2.1]{Ato20}.

To calculate the restrictions to $\mathrm{U}(1)$, we use the surjective homomorphism
\begin{equation}\label{eq:sl2-u11-map}
 \begin{aligned}
  \mathrm{SL}_2(\mathbb{R})\times S^1&\longrightarrow\mathrm{U}(1,1),\\
  (A,z)&\longmapsto zCAC^{-1},
 \end{aligned}
 \qquad
 C=\frac1{\sqrt2}
   \begin{pmatrix}1&i\\1&-i\end{pmatrix}.
\end{equation}
Its kernel is $\{(I,1),(-I,-1)\}$, and hence
\[
 \mathrm{U}(1,1)\simeq
 \bigl(\mathrm{SL}_2(\mathbb{R})\times S^1\bigr)/\{(I,1),(-I,-1)\}.
\]
The standard subgroup $\mathrm{U}(1)\subset\mathrm{U}(1,1)$, given by
$z\mapsto\operatorname{diag}(z,1)$, lifts as
\[
 e^{i\theta}\longmapsto
 \left(
  \begin{pmatrix}
   \cos(\theta/2)&-\sin(\theta/2)\\
   \sin(\theta/2)& \cos(\theta/2)
  \end{pmatrix},
  e^{\frac{i\theta}{2}}
 \right).
\]
Consequently, the $\mathrm{U}(1)$-weights are
\begin{align}
 \left.D_n^+\boxtimes\chi_\ell\right|_{\mathrm{U}(1)}:
 &\quad \frac{n+\ell}{2},\frac{n+\ell}{2}+1,
        \frac{n+\ell}{2}+2,\ldots,
 \label{eq:weight-D-plus}\\
 \left.D_n^-\boxtimes\chi_\ell\right|_{\mathrm{U}(1)}:
 &\quad \frac{-n+\ell}{2},\frac{-n+\ell}{2}-1,
        \frac{-n+\ell}{2}-2,\ldots,
 \label{eq:weight-D-minus}\\
 \left.F_{n,\ell}\right|_{\mathrm{U}(1)}:
 &\quad \frac{-n+\ell}{2}+1,
        \frac{-n+\ell}{2}+2,\ldots,
        \frac{n+\ell}{2}-1.
 \label{eq:weight-compact}
\end{align}
In each case, the multiplicity is one.  Hence
\begin{equation}\label{eq:weight-criterion}
 \operatorname{Hom}_{\mathrm{U}(1)}(\pi_2\boxtimes\chi_m,\mathbb{C})\ne0
 \quad\Longleftrightarrow\quad
 -m\text{ is a $\mathrm{U}(1)$-weight of }\pi_2.
\end{equation}

Let $e_+=(-1,1)$ and $e_-=(1,-1)$ be the generators of
$A_{\phi_2}$ corresponding respectively to
$\chi_{\ell+n-1}$ and $\chi_{\ell-n+1}$, and let $e_1=-1$ be the generator
of $A_{\phi_1}$.  There are three possibilities.

First, suppose that
\[
 m\ge\frac{n-\ell}{2}.
\]
By \eqref{eq:weight-D-minus}, the unique distinguished pair is
\[
 (\chi_m,D_n^-\boxtimes\chi_\ell)
 \quad\text{on}\quad
 (\mathrm{U}(1,0),\mathrm{U}(1,1)).
\]
The corresponding character of $A_{\phi_2}\times A_{\phi_1}$ is
\begin{equation}\label{eq:character-discrete-minus}
 \eta(e_+,1)=-1,\qquad
 \eta(e_-,1)=-1,\qquad
 \eta(1,e_1)=1.
\end{equation}
Indeed,
\begin{align*}
 &\varepsilon\left(\frac12,
   \chi_{\ell+n-1}\otimes\chi_{2m},\psi_{\mathbb C}\right)=-1,
 \varepsilon\left(\frac12,
   \chi_{\ell-n+1}\otimes\chi_{2m},\psi_{\mathbb C}\right)=-1,\\
 &\varepsilon\left(\frac12,
   (\chi_{\ell+n-1}\oplus\chi_{\ell-n+1})
       \otimes\chi_{2m},\psi_{\mathbb C}\right)=1.
\end{align*}
Here both exponents $\ell+n-1+2m$ and $\ell-n+1+2m$ are positive odd
integers.

Second, suppose that
\[
 m\le-\frac{n+\ell}{2}.
\]
Then the unique distinguished pair is
\[
 (\chi_m,D_n^+\boxtimes\chi_\ell)
 \quad\text{on}\quad
 (\mathrm{U}(1,0),\mathrm{U}(1,1)).
\]
The corresponding character is
\begin{equation}\label{eq:character-discrete-plus}
 \eta(e_+,1)=1,\qquad
 \eta(e_-,1)=1,\qquad
 \eta(1,e_1)=1.
\end{equation}
In this range, both exponents $\ell+n-1+2m$ and
$\ell-n+1+2m$ are negative, so all three epsilon values are equal to $1$.

Finally, suppose that
\[
 -\frac{n+\ell}{2}<m<\frac{n-\ell}{2}.
\]
By \eqref{eq:weight-compact}, the unique distinguished pair is
\[
 (\chi_m,F_{n,\ell})
 \quad\text{on}\quad
 (\mathrm{U}(0,1),\mathrm{U}(0,2)).
\]
Its component-group character is
\begin{equation}\label{eq:character-discrete-compact}
 \eta(e_+,1)=-1,\qquad
 \eta(e_-,1)=1,\qquad
 \eta(1,e_1)=-1.
\end{equation}
Indeed, $\ell+n-1+2m>0$ and $\ell-n+1+2m<0$, and therefore
\begin{align*}
 &\varepsilon\left(\frac12,
   \chi_{\ell+n-1}\otimes\chi_{2m},\psi_{\mathbb C}\right)=-1,
 \varepsilon\left(\frac12,
   \chi_{\ell-n+1}\otimes\chi_{2m},\psi_{\mathbb C}\right)=1,\\
 &\varepsilon\left(\frac12,
   (\chi_{\ell+n-1}\oplus\chi_{\ell-n+1})
       \otimes\chi_{2m},\psi_{\mathbb C}\right)=-1.
\end{align*}
Thus the epsilon-dichotomy identity holds for the discrete parameter.

\smallskip
\noindent\emph{The limits of discrete series.}
We now consider
\[
 \phi_2=\chi_\ell\oplus\chi_\ell,
 \qquad \ell\equiv1\pmod2.
\]
Here $A_{\phi_2}\simeq\mathbb Z/2\mathbb Z$, and the packet consists
of the two limits of discrete series $D_1^+\boxtimes\chi_\ell$ and
$D_1^-\boxtimes\chi_\ell$ on $\mathrm{U}(1,1)$; there is no
representation on a compact pure inner form.  These packet descriptions
follow from \cite[Theorem~2.1]{Ato20}.

Let $e_\ell$ be the generator of $A_{\phi_2}$.  The
$\mathrm{U}(1)$-weights of the two packet members are
\begin{align*}
 \left.D_1^+\boxtimes\chi_\ell\right|_{\mathrm{U}(1)}:
 &\quad \frac{\ell+1}{2},\frac{\ell+1}{2}+1,\ldots,\\
 \left.D_1^-\boxtimes\chi_\ell\right|_{\mathrm{U}(1)}:
 &\quad \frac{\ell-1}{2},\frac{\ell-1}{2}-1,\ldots.
\end{align*}
If
\[
 m\le-\frac{\ell+1}{2},
\]
then the distinguished pair is
$(\chi_m,D_1^+\boxtimes\chi_\ell)$ on
$(\mathrm{U}(1,0),\mathrm{U}(1,1))$, and
\[
 \eta(e_\ell,1)=1,\qquad \eta(1,e_1)=1.
\]
This agrees with
\begin{align*}
 \varepsilon\left(\frac12,\chi_\ell\otimes\chi_{2m},
                  \psi_{\mathbb C}\right)&=1,\qquad
 \varepsilon\left(\frac12,(\chi_\ell\oplus\chi_\ell)
                  \otimes\chi_{2m},\psi_{\mathbb C}\right)=1.
\end{align*}
If
\[
 m\ge-\frac{\ell-1}{2},
\]
then the distinguished pair is
$(\chi_m,D_1^-\boxtimes\chi_\ell)$ on
$(\mathrm{U}(1,0),\mathrm{U}(1,1))$, and
\[
 \eta(e_\ell,1)=-1,\qquad \eta(1,e_1)=1.
\]
Now
\begin{align*}
 \varepsilon\left(\frac12,\chi_\ell\otimes\chi_{2m},
                  \psi_{\mathbb C}\right)&=-1,\qquad
 \varepsilon\left(\frac12,(\chi_\ell\oplus\chi_\ell)
                  \otimes\chi_{2m},\psi_{\mathbb C}\right)=1.
\end{align*}
Since $\ell$ is odd, these two ranges exhaust all $m\in\mathbb Z$.

\smallskip
\noindent\emph{The principal series.}
The remaining tempered parameters are
\[
 \phi_2=\chi_\ell\oplus\chi_\ell,
 \qquad \ell\equiv0\pmod2,
\]
and
\[
 \phi_2=\chi_{\ell+it}\oplus\chi_{\ell-it},
 \qquad \ell\in\mathbb Z,\quad t>0.
\]
In both cases $A_{\phi_2}=1$.  The corresponding representations on
$\mathrm{U}(1,1)$ are denoted by $I(\ell,0)$ and $I(\ell,t)$,
respectively; see \cite[Theorem~2.1]{Ato20}.

For $I(\ell,0)$ and $I(\ell,t)$, every integral
$\mathrm{U}(1)$-weight occurs with multiplicity one.  Hence, for every
$m\in\mathbb Z$, the distinguished pair is
\[
 (\chi_m,I(\ell,0))\quad\text{or}\quad(\chi_m,I(\ell,t))
\]
on $(\mathrm{U}(1,0),\mathrm{U}(1,1))$.
Since $A_{\phi_2}$ is trivial, it remains only to check the character
of $A_{\phi_1}$.  In the first case,
\[
 \varepsilon\left(\frac12,(\chi_\ell\oplus\chi_\ell)
                  \otimes\chi_{2m},\psi_{\mathbb C}\right)=1.
\]
In the second case, the identity
\[
 \varepsilon\left(\frac12,\xi,\psi_{\mathbb C}\right)
 \varepsilon\left(\frac12,\xi^{c,-1},\psi_{\mathbb C}\right)=1
\]
from \cite[Section~5]{GGP12} gives
\[
 \varepsilon\left(\frac12,(\chi_{\ell+it}\oplus\chi_{\ell-it})
                  \otimes\chi_{2m},\psi_{\mathbb C}\right)=1.
\]
Thus, the character of $A_{\phi_1}$ is trivial, as required for the
representation $\chi_m$ on $\mathrm{U}(1,0)$.  This completes the
verification of the epsilon-dichotomy identity in the basic cases.

\appendix

\end{document}